\documentclass[11pt]{amsart}
  \usepackage[all,2cell,arrow]{xy}
  \usepackage{mathtools}
\usepackage[english]{babel}
  \usepackage{amssymb}
  \usepackage[left=2.5cm,right=2.5cm,top=3cm,bottom=3cm]{geometry}
 \numberwithin{equation}{section}
 \usepackage{xcolor}
\definecolor{darkgreen}{rgb}{0,0.45,0} 
\usepackage[colorlinks,citecolor=darkgreen,linkcolor=darkgreen]{hyperref}
\usepackage{lscape}
\usepackage{pdflscape}
\usepackage{appendix}
\usepackage{microtype}
\usepackage{enumerate,xspace}
\UseAllTwocells
\CompileMatrices
\allowdisplaybreaks

\renewcommand{\phi}{\varphi}

\newcommand{\act}{\triangleright}

\newcommand{\C}{\mathcal{C}}

\newcommand{\ox}{\otimes}
\newcommand{\x}{\times}

  \newtheorem{proposition}{Proposition}[section]
  \newtheorem{lemma}[proposition]{Lemma}
  
  \newtheorem{corollary}[proposition]{Corollary}
  \newtheorem{theorem}[proposition]{Theorem}

  \theoremstyle{definition}
  \newtheorem{definition}[proposition]{Definition}
  
  \newtheorem{examples}[proposition]{Examples}
  \newtheorem{remark}[proposition]{Remark}

  \newcounter{c}
  \renewcommand{\[}{\setcounter{c}{1}$$}
  \newcommand{\etyk}[1]{\vspace{-7.4mm}$$\begin{equation}\Label{#1}
  \addtocounter{c}{1}}
  \renewcommand{\]}{\ifnum \value{c}=1 $$\else \end{equation}\fi}
  \title{$\sf 
  S$-protomodularity of the category of cocommutative bialgebras}

\usepackage{tikz}
\usepackage{tikz-cd}
\usepackage{tikzit}

\tikzstyle{red}=[fill=red, draw=black, shape=circle]
\tikzstyle{green}=[fill=white, draw=white, shape=rectangle, tikzit draw=black]

\tikzstyle{arrow}=[->, yshift=40pt]
\tikzstyle{stealth}=[->, >=stealth]
\tikzstyle{dashouille}=[->, dashed]

\author{Florence Sterck}

\address{Institut de Recherche en Math\'ematique et Physique, Universit\'e catholique de Louvain, Chemin du Cyclotron 2, 1348 
Louvain-la-Neuve \and D\'epartement de Math\'ematique, Universit\'e Libre de Bruxelles, Campus de la Plaine – CP 210
Boulevard du Triomphe
1050 Bruxelles, Belgium }
\email{florence.sterck@uclouvain.be}

\begin{document}
\maketitle

\begin{abstract}
We prove that the category of cocommutative bialgebras in any symmetric monoidal category (that has equalizers) is an $S$-protomodular category with respect to a particular class of split extensions of cocommutative bialgebras. We also obtain the ``partial'' well-known \emph{Smith is Huq} condition, meaning that two $S$-equivalence relations centralize each other as soon as the normal subobjects associated with them commute in the sense of Huq.
\end{abstract}

\section*{Introduction}

The notion of protomodular category was introduced by Bourn in \cite{Bourn}. The categories of groups, rings, Lie algebras, crossed modules, rings, the dual of the category of sets are examples of protomodular categories. 
In the pointed case, a category $\C$ is protomodular if and only if the Split Short Five Lemma holds. This means that for any diagram of the form
 \begin{equation}\label{SSFL}
\begin{tikzpicture}[descr/.style={fill=white},baseline=(A.base),scale=0.8] 
\node (A) at (0,0) {$A_1$};
\node (B) at (2.5,0) {$B_1$};
\node (C) at (-2.5,0) {$X_1$};
\node (A') at (0,-2) {$A_2$};
\node (B') at (2.5,-2) {$B_2$};
\node (C') at (-2.5,-2) {$X_2$};
\node (O1) at (-4.5,0) {$0$};
\node (O1') at (-4.5,-2) {$0$};
\node (O2) at (4.5,0) {$0$};
\node (O2') at (4.5,-2) {$0$};
\path[->,font=\scriptsize]
(O1) edge node[above] {$ $} (C)
(B) edge node[above] {$ $} (O2)
(B') edge node[above] {$ $} (O2')
(O1') edge node[above] {$ $} (C')
(B.south) edge node[right] {$ g$}  (B'.north)
 (C.south) edge node[left] {$ v $}  (C'.north)
(A.south) edge node[left] {$ p$} (A'.north)
(C'.east) edge node[above] {$\kappa_2$} (A'.west)
([yshift=-4pt]A'.east) edge node[below] {$\alpha_2$} ([yshift=-4pt]B'.west)
([yshift=2pt]B'.west) edge node[above] {$e_2$} ([yshift=2pt]A'.east)
(C.east) edge node[above] {$\kappa_1$} (A.west)
([yshift=-4pt]A.east) edge node[below] {$\alpha_1$} ([yshift=-4pt]B.west)
([yshift=2pt]B.west) edge node[above] {$e_1$} ([yshift=2pt]A.east);
\end{tikzpicture} 
 \end{equation}
 where $\kappa_i$ is the kernel of $\alpha_i$ and $\alpha_i \cdot e_i = 1_{B_i}$ for any $i \in \{ 1,2\}$, then $p$ is an isomorphism whenever $v$ and $g$ are, this property is often referred to as the ``Split Short Five Lemma'' holds in $\C$.

It is well-known that the category of internal groups in a finitely complete category is protomodular \cite{Bourn}. In particular, this result implies that the category of cocommutative Hopf algebras in a symmetric monoidal category, that has equalizers, is protomodular, seen as internal groups in the category of cocommutative coalgebras. The category of cocommutative Hopf algebras over a field is even semi-abelian \cite{GSV}. 

In this paper, we are interested in cocommutative bialgebras in any symmetric monoidal category that has equalizers. As a matter of fact, it was proven in \cite{GVdL} that the category of cocommutative $K$-bialgebras is not a protomodular category. 
Similarly, the category of monoids is not a protomodular category. However, a class of split epimorphisms of monoids, called \emph{Schreier split epimorphisms}, turned out to have some very interesting properties, similar to the ones of split epimorphisms of groups. For example, \emph{Schreier split epimorphisms} are equivalent to the actions of monoids, the Split Short Five Lemma holds, etc. \cite{BMS,BMMS}. The authors in \cite{BMS}, introduced the notion of $S$-protomodularity, that is the protomodularity with respect to a class $S$ of split epimorphisms. It implies that we have the Split Short Five Lemma as in \eqref{SSFL} whenever $(\alpha_1,e_1)$ and $(\alpha_2,e_2)$ are in the class $S$.
The main example is the category of monoids with the class of \emph{Schreier split epimorphisms}.

In \cite{Stercksplit}, we defined a notion of split extensions for (non-associative) bialgebras (and non-associative Hopf algebras) such that they have ``group-like'' properties. More precisely, we showed that this definition of split extension of (non-associative) bialgebras is equivalent to the notion of action of (non-associative) bialgebras. Moreover, we proved the validity of the Split Short Five
Lemma for these kinds of split extensions. 
In particular, these results restrict to the case of cocommutative bialgebras.

Hence, it is natural to hope that the category of cocommutative bialgebras in any symmetric monoidal category is $S$-protomodular with respect to the class introduced in \cite{Stercksplit}.

In a finitely complete category, we can define two types of centrality: the centrality of equivalence relations in the sense of Smith and the centrality of normal monomorphisms in the sense of Huq. Note that these notions of centrality are not independent.  If two relations of equivalence centralize each other (in the sense of Smith) then the corresponding normal monomorphisms necessarily centralize in the sense of Huq \cite{BG}.

The converse is not true in general. If $\mathcal{C}$ is a category such that any two equivalence relations always centralize each other as soon as their normalizations centralize in the sense of Huq, one says that $\C$ satisfies the so-called \textit{Smith is Huq} property \cite{BG,MFVdl}. 

In \cite{MFM}, the authors introduced the notion of \emph{Smith is Huq} for pointed $S$-protomodular categories. In that context, the \emph{Smith is Huq} property can be expressed as follows: two $S$-equivalence relations centralize each other if and only if their associated normal subobjects commute (in the sense of Huq). 

In this paper, we prove that the category of cocommutative bialgebras in any symmetric monoidal category is $S$-protomodular with respect to the class of split extensions introduced in \cite{Stercksplit}, we give a description of the centrality in the sense of Huq for two subbialgebras of a cocommutative bialgebra and we examine the \emph{Smith is Huq} condition for cocommutative bialgebras. 
The results obtained in this paper generalize results in \cite{BMMS} and \cite{MFM} on the category of monoids.

The layout of this article is as follows: the first section contains some preliminaries on cocommutative bialgebras in a symmetric monoidal category.

In the second section, we prove that the category of cocommutative bialgebras in any symmetric monoidal category (that has equalizers) is $S_{coc}$-protomodular with respect to the class $S_{coc}$ of split extensions of cocommutative bialgebras defined in \cite{Stercksplit}.

In the third section, we describe the Huq commutator of two subbialgebras of a cocommutative bialgebra.

In the last section, we investigate the \emph{Smith is Huq} condition for cocommutative bialgebras by using the results of \cite{MFM}.

\section{Preliminaries}

\subsection{Bialgebras in a symmetric monoidal category}
We recall that a \textit{monoidal category} is given by a triple $(\mathcal{C}, \otimes, I)$ where $\mathcal{C}$ is a category, $\otimes \colon \mathcal{C} \times \mathcal{C} \rightarrow \mathcal{C}$ a bifunctor and $I$ is the identity element (we omit to explicit the three natural isomorphisms, the associator, the right unit and the left unit).

A \textit{braided monoidal category} is a 4-tuple  $(\mathcal{C}, \otimes, I, \sigma)$ where  $(\mathcal{C}, \otimes, I)$ is a monoidal category and $\sigma$ is a \textit{braiding}. A braiding consists of a family of natural isomorphisms $\sigma_{X,Y} \colon X \otimes Y \rightarrow Y \otimes X$ satisfying \[ \sigma_{X \ox Y, Z} = (\sigma_{X,Z} \ox 1_Y)\cdot (1_X \ox \sigma_{Y,Z} )\]
 \[ \sigma_{X , Y \ox Z} = (1_Y \ox \sigma_{X,Z} )\cdot (\sigma_{X,Y} \ox 1_Z).\]
 A braided monoidal category is called \textit{symmetric} when \begin{equation}\label{symm}
 \sigma_{Y,X}^{-1} = \sigma_{X,Y}.
 \end{equation}
In this paper, we omit the indexes of the braiding when this does not bring any confusion.

An \textit{algebra} in a symmetric monoidal category $(\mathcal{C}, \otimes, I, \sigma)$ is given by an object $A \in \mathcal{C}$ endowed with a morphism $m \colon A \ox A \rightarrow A$, called the multiplication. An algebra is associative and unital when there is a morphism $u_A \colon I \rightarrow A$ called the unit, such that the following equalities are satisfied
\begin{equation}\label{unital multiplication}
m \cdot (u_A \ox 1_A) = 1_A = m \cdot (1_A \ox u_A)
\end{equation}
\begin{equation}\label{ass}
 m \cdot (m \ox 1_A)  = m \cdot (1_A \ox m) 
\end{equation}  
\begin{center}

\begin{tikzpicture}[descr/.style={fill=white},baseline=(A.base)] 
\node (A) at (0,0) {$ A \ox A$};
\node (B) at (2,0) {$A$};
\node (C) at (-2,0) {$A $};
\node (D) at (0,-1.5) {$ A.$};
\draw[commutative diagrams/.cd, ,font=\scriptsize]
(B.south west) edge[commutative diagrams/equal] (D.north east)
(C.south east) edge[commutative diagrams/equal] (D.north west);
\path[->,font=\scriptsize]
(A.south) edge node[descr] {$m$} (D.north)
(C.east) edge node[above] {$ u_A \ox 1_A$} (A.west)
(B.west) edge node[above] {$1_A \ox u_A$} (A.east)
;
\end{tikzpicture} 
\begin{tikzpicture}[descr/.style={fill=white},baseline=(D.base)] 
\node (A) at (0,0) {$ A \ox A $};
\node (B) at (3,0) {$A.$};
\node (C) at (3,1.5) {$A \ox A $};
\node (D) at (0,1.5) {$ A \ox A \ox A$};
\path[->,font=\scriptsize]
(C.south) edge node[right] {$m $} (B.north)
(D.south) edge node[descr] {$1_A \ox m$} (A.north)
(A.east) edge node[above] {$ m $} (B.west)
(D.east) edge node[above] {$ m \ox 1_A$} (C.west)
;
\end{tikzpicture}
\end{center}
All the algebras that we will consider in this paper are associative and unital.
A\textit{ morphism of algebras} $f \colon A \rightarrow B$ is a morphism in $\mathcal{C}$ such that the following diagrams commute
\begin{center}
\begin{tikzpicture}[descr/.style={fill=white},baseline=(D.base)] 
\node (A) at (0,0) {$ A $};
\node (B) at (3,0) {$B$};
\node (C) at (3,1.5) {$B \ox B $};
\node (D) at (0,1.5) {$ A \ox A$};
\path[->,font=\scriptsize]
(C.south) edge node[descr] {$m$} (B.north)
(D.south) edge node[descr] {$m$} (A.north)
(A.east) edge node[above] {$ f$} (B.west)
(D.east) edge node[above] {$ f \ox f$} (C.west)
;
\end{tikzpicture} 
\begin{tikzpicture}[descr/.style={fill=white},baseline=(A.base)] 
\node (A) at (0,0) {$ A $};
\node (C) at (-2,0) {$I $};
\node (D) at (0,-1.5) {$ B.$};
\path[->,font=\scriptsize]
(A.south) edge node[descr] {$f $} (D.north)
(C.south east) edge node[descr] {$u_B $} (D.north west)
(C.east) edge node[above] {$ u_A $} (A.west)
;
\end{tikzpicture} 
\end{center}
A coalgebra is the dual notion of the notion of an algebra. In other words, a coalgebra over $(\mathcal{C}, \otimes, I, \sigma)$ is an object $C \in \mathcal{C}$ with a comultiplication $\Delta \colon C \rightarrow C \ox C$.  From now on, the coalgebras will always be coassociative, i.e. the following equality holds
\begin{equation}\label{coass comultiplication}
(\Delta \ox 1_C) \cdot \Delta = (1_C \ox \Delta) \cdot \Delta
\end{equation}
\begin{center}
\begin{tikzpicture}[descr/.style={fill=white},baseline=(D.base)] 
\node (A) at (0,0) {$ C \ox C $};
\node (B) at (3,0) {$C \ox C \ox C.$};
\node (C) at (3,1.5) {$C \ox C $};
\node (D) at (0,1.5) {$ C$};
\path[->,font=\scriptsize]
(C.south) edge node[descr] {$1_C \ox \Delta$} (B.north)
(D.south) edge node[descr] {$\Delta$} (A.north)
(A.east) edge node[above] {$ \Delta \ox 1_C$} (B.west)
(D.east) edge node[above] {$ \Delta$} (C.west)
;
\end{tikzpicture}
\end{center}
We will also assume that the coalgebras are counital, meaning that there exists a morphism $\epsilon_C \colon C \rightarrow I$, called counit, satisfying the condition:
\begin{equation}\label{counital comultiplication}
( \epsilon_C \ox 1_C) \cdot \Delta = 1_C = (1_C  \ox \epsilon_C) \cdot \Delta,
\end{equation}
as expressed  by the commutativity of the following diagram
\begin{center}
\begin{tikzpicture}[descr/.style={fill=white},baseline=(A.base)] 
\node (A) at (0,0) {$ C \ox C$};
\node (B) at (2,0) {$C$};
\node (C) at (-2,0) {$C $};
\node (D) at (0,-1.5) {$ C$};
\draw[commutative diagrams/.cd, ,font=\scriptsize]
(B.south west) edge[commutative diagrams/equal] (D.north east)
(C.south east) edge[commutative diagrams/equal] (D.north west);
\path[->,font=\scriptsize]
(D.north) edge node[descr] {$\Delta$} (A.south)
(A.west) edge node[above] {$  \epsilon_C \ox 1_C$} (C.east)
(A.east) edge node[above] {$1_C \ox \epsilon_C$} (B.west)
;
\end{tikzpicture} 
\end{center}
Similarly, a \textit{morphism of coalgebras} $g \colon C \rightarrow D$ is  a morphism in $\mathcal{C}$ such that the following two diagrams commute
\begin{center}
\begin{tikzpicture}[descr/.style={fill=white},baseline=(D.base)] 
\node (A) at (0,0) {$ D $};
\node (B) at (3,0) {$C$};
\node (C) at (3,1.5) {$C \ox C $};
\node (D) at (0,1.5) {$ D \ox D$};
\path[->,font=\scriptsize]
(B.north) edge node[descr] {$\Delta$} (C.south) 
 (A.north) edge node[descr] {$\Delta$} (D.south)
 (B.west) edge node[above] {$ g$} (A.east)
(C.west) edge node[above] {$ g\ox g$} (D.east)
;
\end{tikzpicture} 
\begin{tikzpicture}[descr/.style={fill=white},baseline=(A.base)] 
\node (A) at (0,0) {$ D $};
\node (C) at (-2,0) {$I $};
\node (D) at (0,-1.5) {$ C.$};
\path[->,font=\scriptsize]
(D.north)edge node[descr] {$g $} (A.south) 
(D.north west) edge node[descr] {$\epsilon_C $} (C.south east)
(A.west) edge node[above] {$ \epsilon_D $} (C.east)
;
\end{tikzpicture} 
\end{center}

We also recall that a \textit{bialgebra} is a 5-tuple $(B,m,u_B,\Delta,\epsilon_B)$ where $(B,m,u_B)$ is an algebra, $(B,\Delta, \epsilon_B)$ is a coalgebra and $\Delta, \epsilon_B$ are algebra morphisms (which is equivalent of asking that $m$, $u_B$ are coalgebra morphisms) i.e.
\begin{equation}\label{delta et m}
 \Delta \cdot m = (m \ox m) \cdot (1_B \ox \sigma \ox 1_B) \cdot (\Delta \ox \Delta)
 \end{equation}
 \begin{equation}\label{delta et u}
 \Delta \cdot u_B = u_B \ox u_B
 \end{equation}
 \begin{equation}\label{epsilon et m}
 \epsilon_B \cdot m = \epsilon_B \ox \epsilon_B
 \end{equation}
 \begin{equation}\label{epsilon et u}
 \epsilon_B \cdot u_B = 1_I
\end{equation}
Moreover, a morphism in $\mathcal{C}$ is a \textit{morphism of bialgebras} if it is a morphism of algebras and coalgebras.

A bialgebra is called cocommutative when its underlying coalgebra structure is cocommutative, it means that \begin{equation}\label{coco}
\sigma \cdot \Delta = \Delta.
\end{equation}  

The category of cocommutative bialgebras in $\C$, a symmetric monoidal category, is denoted by $\sf Bial_{\C,coc}$.

\begin{examples}
(1) In the symmetric monoidal category $(\mathsf{Set}, \times , \{\star\})$ of sets where $\sigma$ is the twist morphism ( where $\sigma(x,y) = (y,x)$ for any element $x$ of a set $X$ and any element $y$ of a set  $Y$), every object has a coalgebra structure with $\Delta$ being the diagonal and $\epsilon$ the morphism sending every element to the singleton. Hence, a (cocommutative) bialgebra (or algebra) is a monoid.

(2) In the symmetric monoidal category $(\mathsf{Vect_K},\ox,K)$ of vector spaces over a field $K$ where $\sigma$ is the twist morphism (defined by $\sigma(x \ox y) = y \ox x$ for any $x \ox y \in X \ox Y$ ), we recover the notion of $K$-algebra, $K$-coalgebra and $K$-bialgebra.

(3) In \cite{CG}, a symmetric monoidal category was introduced such that Hom-algebras, Hom-coalgebras and Hom-bialgebras (see \cite{MS}) coincide with the algebras, coalgebras and bialgebras in this symmetric monoidal category. 

\end{examples}

\subsection{Adjunction cocommutative bialgebras and cocommutative coalgebras}

It is well-known that we have the following adjunction between cocommutative bialgebras and cocommutative coalgebras.

\begin{equation}\label{adj}
\begin{tikzpicture}[descr/.style={fill=white},scale=1.2,baseline=(A.base)]
\node at (1.5,0) {$\perp$};
\node (A) at (0,0) {$\mathsf{BiAlg_{\C,coc}}$};
\node (B) at (3,0) {$\mathsf{CoAlg}_{\C,coc}$};
 \path[->,font=\scriptsize]
 ([yshift=5pt]B.west) edge node[above] {$F$}([yshift=5pt]A.east) 
([yshift=-5pt]A.east) edge node[below] {$U$} ([yshift=-5pt]B.west);
\end{tikzpicture}
\end{equation}
where $U$ is the forgetful functor and $F$ is the free algebra functor. 

In particular, this adjunction implies that $U$ preserves the limits and that a monomorphism of bialgebras is in particular also a monomorphism of coalgebras. This observation will be useful several times in this paper.

\subsection{Limits in the category of cocommutative bialgebras}

We recall the constructions of products, equalizers and pullbacks in the category of cocommutative bialgebras in a symmetric monoidal category that has equalizers.
It is interesting to recall that the categorical product of cocommutative bialgebras is given by the monoidal product. 

\begin{proposition}\label{tensor product}
In the category $\sf BiAlg_{\C,coc}$, the categorical product of two cocommutative bialgebras $A$ and $B$ is given by the monoidal product
$(A \otimes B, \pi_A,\pi_B)$ where the projections $\pi_A \colon A \otimes B \rightarrow A$ and $\pi_B \colon A \otimes B \rightarrow B$ are defined by $\pi_A \coloneqq 1_A \otimes \epsilon_B $ and $\pi_B \coloneqq \epsilon_A \otimes 1_B$.

\begin{equation}\label{triangleproduit}
\begin{tikzpicture}[descr/.style={fill=white},baseline=(A.base),scale=1.3] 
\node (A) at (0,0) {$ A \otimes B$};
\node (B) at (2,0) {$B$};
\node (C) at (-2,0) {$A $};
\node (D) at (0,-1.5) {$ C.$};
\path[->,font=\scriptsize]
  (D.north east)edge node[right,xshift=5pt] {$g$}(B.south west)
  (D.north west)edge node[left,xshift=-5pt] {$f$}(C.south east)
  (A.west)edge node[above] {$\pi_A$}(C.east)
 (A.east)edge node[above] {$\pi_B$} (B.west)
;
\path[->,font=\scriptsize,dashed]
(D.north) edge node[descr] {$(f \ox g) \cdot \Delta$}(A.south)
;
\end{tikzpicture}
\end{equation}
\end{proposition}
From now on, we are considering a symmetric monoidal category $\C$ that has equalizers.

The equalizers are defined as in \cite{Agore}.
Let $f,g \colon A \to B$ be two morphisms of bialgebras, then the following construction in $\C$ is the equalizer of $f$ and $g$. 

 \begin{equation}\label{equalizer}
\begin{tikzpicture}[descr/.style={fill=white},baseline=(A.base)] 
\node (A) at (-1,0) {$A$};
\node (B) at (2.5,0) {$B \ox A $};
\node (C) at (-2.5,0) {$E$};
\path[->,font=\scriptsize]
([yshift=-4pt]A.east) edge node[below] {$(f \ox 1_A) \cdot \Delta$} ([yshift=-4pt]B.west)
([yshift=0pt]C.east) edge node[above] {$\varepsilon $} ([yshift=0pt]A.west)
([yshift=4pt]A.east) edge node[above] {$( g \ox 1_A) \cdot \Delta$} ([yshift=4pt]B.west)
;
\end{tikzpicture}.
\end{equation}
where $\varepsilon$ is the equalizer of $(f \otimes 1_A) \cdot \Delta$ and $( g \ox 1_A) \cdot \Delta$ in $\C$. We can easily check that this construction is the equalizer of $f$ and $g$ in $\sf BiAlg_{\C,coc}$.

Via the definitions of products and equalizers in $\sf BiAlg_{\C,coc}$, we define the pullback in $\sf BiAlg_{\C,coc}$ of two morphisms of bialgebras $f \colon A \to B$ and $g \colon C \to B$ as follows

\[
\begin{tikzpicture}[descr/.style={fill=white},baseline=(current  bounding  box.center),xscale=1.3] 
\node (A) at (0,0) {$A \otimes_B C$};
\node (B) at (2.5,0) {$C$};
\node (A') at (0,-2) {$A$};
\node (B') at (2.5,-2) {$B$};
\path[->,font=\scriptsize]
(B) edge node[right] {$ g$}  (B')
(A.south) edge node[left] {$ p_A$} (A'.north)
(A'.east) edge node[below] {$f$} (B'.west)
(A.east) edge node[above] {$p_C$} (B.west)
;
\end{tikzpicture} \]
where $A \otimes_B C$ is the object in the following equalizer
\[ 
\begin{tikzpicture}[descr/.style={fill=white},baseline=(A.base),xscale=1.5] 
\node (A) at (-1,0) {$A \otimes C$};
\node (B) at (2.5,0) {$A \ox B \ox C $};
\node (C) at (-2.5,0) {$A \otimes_B C$};
\path[->,font=\scriptsize]
([yshift=-4pt]A.east) edge node[below] {$(1_A \ox f \ox 1_C) \cdot ( \Delta \ox 1_C) $} ([yshift=-4pt]B.west)
([yshift=0pt]C.east) edge node[above] {$\varepsilon$} ([yshift=0pt]A.west)
([yshift=4pt]A.east) edge node[above] {$(1_A \otimes g \ox 1_C) \cdot (1_A \otimes \Delta)  $} ([yshift=4pt]B.west)
;
\end{tikzpicture}.
\]
and the two projections are defined as $p_A \coloneqq (1_A \otimes \epsilon_C)\cdot \varepsilon $ and $p_C \coloneqq (\epsilon_A \otimes 1_C) \cdot \varepsilon$.
\subsection{Split extensions of cocommutative bialgebras}

In \cite{Stercksplit}, we introduced a notion of split extensions of (non-associative) bialgebras and we showed several properties. Here, we recall this notion and some results in the case of cocommutative bialgebras that will be useful later on. 

\begin{definition}\label{definition split extension}
A \textit{split extension of cocommutative bialgebras} is given by a diagram \begin{equation}\label{split extension}
\begin{tikzpicture}[descr/.style={fill=white},baseline=(A.base)] 
\node (A) at (0,0) {$A$};
\node (B) at (2.5,0) {$B$};
\node (C) at (-2.5,0) {$X$};
\path[dashed,->,font=\scriptsize]
([yshift=2pt]A.west) edge node[above] {$\lambda$} ([yshift=2pt]C.east);
\path[->,font=\scriptsize]
([yshift=-4pt]C.east) edge node[below] {$\kappa$} ([yshift=-4pt]A.west)
([yshift=-4pt]A.east) edge node[below] {$\alpha$} ([yshift=-4pt]B.west)
([yshift=2pt]B.west) edge node[above] {$e$} ([yshift=2pt]A.east);
\end{tikzpicture} ,
\end{equation}
where $X$, $A$, $B$ are cocommutative bialgebras, $\kappa$, $\alpha$, $e$ are morphisms of bialgebras, such that
\begin{itemize}
\item[(1)] $\lambda \cdot \kappa = 1_X$, $\alpha \cdot e =1_B$ ,
\item[(2)] $\lambda \cdot e = u_X\cdot \epsilon_B$, $\alpha \cdot \kappa = u_B \cdot\epsilon_X$,
\item[(3)] $m \cdot ((\kappa\cdot \lambda) \otimes (e \cdot \alpha)) \cdot \Delta = 1_A$,
\item[(4)]$ \lambda \cdot m \cdot (\kappa \otimes e) =  1_X \otimes \epsilon_B$,
\item[(5)] $\lambda$ is a morphism of coalgebras preserving the unit.
\end{itemize} 
\end{definition}
We recall that the conditions $\lambda \cdot \kappa = 1_X$, $\lambda \cdot e = u_X \cdot \epsilon_B$ and the preservation of the unit by $\lambda$ are consequences of the axiom $(4)$. 
\begin{definition}\label{morph split ext}
A \textit{morphism of split extensions} from the split extension
\begin{tikzpicture}[descr/.style={fill=white},baseline=(A.base),xscale=0.7] 
\node (A) at (0,0) {$A$};
\node (B) at (2.5,0) {$B$};
\node (C) at (-2.5,0) {$X$};
\path[dashed,->,font=\scriptsize]
([yshift=2pt]A.west) edge node[above] {$\lambda$} ([yshift=2pt]C.east);
\path[->,font=\scriptsize]
([yshift=-4pt]C.east) edge node[below] {$\kappa$} ([yshift=-4pt]A.west)
([yshift=-4pt]A.east) edge node[below] {$\alpha$} ([yshift=-4pt]B.west)
([yshift=2pt]B.west) edge node[above] {$e$} ([yshift=2pt]A.east);
\end{tikzpicture} 
 to the split extension
\begin{tikzpicture}[descr/.style={fill=white},baseline=(A.base),xscale=0.7] 
\node (A) at (0,0) {$A'$};
\node (B) at (2.5,0) {$B'$};
\node (C) at (-2.5,0) {$X'$};
\path[dashed,->,font=\scriptsize]
([yshift=2pt]A.west) edge node[above] {$\lambda'$} ([yshift=2pt]C.east);
\path[->,font=\scriptsize]
([yshift=-4pt]C.east) edge node[below] {$\kappa'$} ([yshift=-4pt]A.west)
([yshift=-4pt]A.east) edge node[below] {$\alpha'$} ([yshift=-4pt]B.west)
([yshift=2pt]B.west) edge node[above] {$e'$} ([yshift=2pt]A.east);
\end{tikzpicture} 
is given by 3 morphisms of bialgebras $g\colon B \rightarrow B'$, $v \colon X \rightarrow X'$ and $p \colon A \rightarrow A'$ such that the following diagram commutes \begin{center}
\begin{tikzpicture}[descr/.style={fill=white},baseline=(A.base),xscale=0.7] 
\node (A) at (0,0) {$A$};
\node (B) at (2.5,0) {$B$};
\node (C) at (-2.5,0) {$X$};
\node (A') at (0,-2) {$A'$};
\node (B') at (2.5,-2) {$B'.$};
\node (C') at (-2.5,-2) {$X'$};
\path[dashed,->,font=\scriptsize]
([yshift=2pt]A.west) edge node[above] {$\lambda$} ([yshift=2pt]C.east)
([yshift=2pt]A'.west) edge node[above] {$\lambda'$} ([yshift=2pt]C'.east);
\path[->,font=\scriptsize]
(B.south) edge node[right] {$ g$}  (B'.north)
 (C.south) edge node[left] {$ v $}  (C'.north)
(A.south) edge node[left] {$ p$} (A'.north)
([yshift=-4pt]C'.east) edge node[below] {$\kappa'$} ([yshift=-4pt]A'.west)
([yshift=-4pt]A'.east) edge node[below] {$\alpha'$} ([yshift=-4pt]B'.west)
([yshift=2pt]B'.west) edge node[above] {$e'$} ([yshift=2pt]A'.east)
([yshift=-4pt]C.east) edge node[below] {$\kappa$} ([yshift=-4pt]A.west)
([yshift=-4pt]A.east) edge node[below] {$\alpha$} ([yshift=-4pt]B.west)
([yshift=2pt]B.west) edge node[above] {$e$} ([yshift=2pt]A.east);
\end{tikzpicture} 
\end{center}
\end{definition}

The two definitions above give rise to the category $\mathsf{SplitExt(BiAlg_{\C,coc})}$ of split extensions of bialgebras. In this paper, we will denote this class of split extensions of cocommutative bialgebras by $S_{coc}$.

We recall a convenient equality. 
\begin{lemma}
Let 
\begin{tikzpicture}[descr/.style={fill=white},baseline=(A.base)] 
\node (A) at (0,0) {$A$};
\node (B) at (2.5,0) {$B$};
\node (C) at (-2.5,0) {$X$};
\path[dashed,->,font=\scriptsize]
([yshift=2pt]A.west) edge node[above] {$\lambda$} ([yshift=2pt]C.east);
\path[->,font=\scriptsize]
([yshift=-4pt]C.east) edge node[below] {$\kappa$} ([yshift=-4pt]A.west)
([yshift=-4pt]A.east) edge node[below] {$\alpha$} ([yshift=-4pt]B.west)
([yshift=2pt]B.west) edge node[above] {$e$} ([yshift=2pt]A.east);
\end{tikzpicture} 
be a split extension of bialgebras, then the following identities hold,
 \begin{equation}\label{lambda morph Prop 2.5}
 \lambda \cdot m = m \cdot (\lambda \otimes \lambda) \cdot (1_A \otimes m) \cdot (1_A \otimes (e \cdot \alpha) \ox (\kappa \ox \lambda)) \cdot (\Delta \otimes 1_A),
 \end{equation}
 where $\act = \lambda \cdot m \cdot (e \ox \kappa)$.
 \end{lemma}

Even if the category of cocommutative bialgebras is not protomodular (see \cite{GVdL}), we have an interesting result for the split extensions defined above: a relative form of the Split Short Five Lemma holds in $\sf BiAlg_{\C,coc}$. 

\begin{theorem}\label{Sec SSFL}
Let $(g,v,p)$ be a morphism of split extensions of bialgebras in a symmetric monoidal category $\C$ 
\begin{center}
\begin{tikzpicture}[descr/.style={fill=white}] 
\node (A) at (0,0) {$A$};
\node (B) at (2.5,0) {$B$};
\node (C) at (-2.5,0) {$X$};
\node (A') at (0,-2) {$A'$};
\node (B') at (2.5,-2) {$B'$};
\node (C') at (-2.5,-2) {$X'$};
\path[dashed,->,font=\scriptsize]
([yshift=2pt]A'.west) edge node[above] {$\lambda'$} ([yshift=2pt]C'.east)
([yshift=2pt]A.west) edge node[above] {$\lambda$} ([yshift=2pt]C.east)
;
\path[->,font=\scriptsize]
(B.south) edge node[right] {$ g$}  (B'.north)
 (C.south) edge node[left] {$ v $}  (C'.north)
(A.south) edge node[left] {$ p$} (A'.north)
([yshift=-4pt]C'.east) edge node[below] {$\kappa'$} ([yshift=-4pt]A'.west)
([yshift=-4pt]A'.east) edge node[below] {$\alpha'$} ([yshift=-4pt]B'.west)
([yshift=2pt]B'.west) edge node[above] {$e'$} ([yshift=2pt]A'.east)
([yshift=-4pt]C.east) edge node[below] {$\kappa$} ([yshift=-4pt]A.west)
([yshift=-4pt]A.east) edge node[below] {$\alpha$} ([yshift=-4pt]B.west)
([yshift=2pt]B.west) edge node[above] {$e$} ([yshift=2pt]A.east);
\end{tikzpicture} 
\end{center}
 then $p$ is an isomorphism whenever $v$ and $g$ are.
\end{theorem}

In \cite{Stercksplit}, it was proved that there is an equivalence between the categories of actions and the one of split extensions of (cocommutative) bialgebras in a symmetric monoidal category. We recall this equivalence and the definition of the objects and morphisms of the category of actions of bialgebras.

\begin{definition}\label{def action}
Let $X$ and $B$ be cocommutative bialgebras in a symmetric monoidal category $(\mathcal{C}, \otimes, I, \sigma)$. An \textit{action of bialgebras} is a morphism in $\mathcal{C}$, $\act \colon B \otimes X \rightarrow X $, such that
\begin{align*}
&\triangleright \cdot (u_B \otimes 1_X)  = 1_X,\\
&\act \cdot (1_B \ox \act) = \act \cdot (m \ox 1_X) ,\\
& \triangleright \cdot (1_B \otimes u_X) =  u_X \cdot\epsilon_B,\\
&\act \cdot (1_B \ox m) = m \cdot (\act \ox \act) \cdot (1_B \ox \sigma \ox 1_X) \cdot (\Delta \ox 1_X \ox 1_X),\\
&\epsilon_X \cdot \triangleright = \epsilon_B \otimes \epsilon_X,\\
&\Delta \cdot  \triangleright = (\triangleright \otimes \triangleright) \cdot (1_B \otimes 
 \sigma \otimes 1_X) \cdot (\Delta \otimes \Delta).
\end{align*}
\end{definition}
In other words, an action of bialgebras of $B$ on $X$ is a cocommutative bialgebra in the symmetric monoidal category of $B$-modules, where $B$ is a cocommutative bialgebra.
\begin{definition}\label{morph action}
Let $\act \colon B \otimes X \rightarrow X$ and $\act' \colon B' \otimes X' \rightarrow X'$ be two actions of bialgebras. A \textit{morphism of actions of bialgebras} is defined as a pair of morphisms of bialgebras $g \colon B \rightarrow B'$ and $v \colon X \rightarrow X'$ such that  \[v \cdot \act = \act' \cdot (g \ox v).\]
\end{definition}
There is then the category of actions of cocommutative bialgebras that will be denoted by  $\sf Act(BiAlg_{\C,coc})$.

We recall the theorem obtained in \cite{Stercksplit} and the construction of the functor in the equivalence of categories between $\sf Act(BiAlg_{\C,coc})$ and $\sf SplitExt(BiAlg_{\C,coc})$.

\begin{theorem}\label{equi bialg}
There is an equivalence between the category $\mathsf{SplitExt(BiAlg_{\C,coc})}$ of split extensions of cocommutative bialgebras and the category $\mathsf{Act(BiAlg_{\C,coc})}$ of actions of cocommutative bialgebras.
\end{theorem}

\begin{proof}
The functor $ F \colon \mathsf{SplitExt(BiAlg_{\C,coc})} \rightarrow \mathsf{Act(BiAlg_{\C,coc})}$ is defined as

\[ F \left( 
\begin{tikzpicture}[descr/.style={fill=white},baseline=(O.base)] 
\node (O) at (1,-1) {$ $};
\node (A) at (0,0) {$A$};
\node (B) at (2.5,0) {$B$};
\node (C) at (-2.5,0) {$X$};
\node (A') at (0,-2) {$A'$};
\node (B') at (2.5,-2) {$B'$};
\node (C') at (-2.5,-2) {$X'$};
\path[dashed,->,font=\scriptsize]
([yshift=2pt]A'.west) edge node[above] {$\lambda'$} ([yshift=2pt]C'.east)
([yshift=2pt]A.west) edge node[above] {$\lambda$} ([yshift=2pt]C.east)
;
\path[->,font=\scriptsize]
(C.south) edge node[left] {$ v$} (C'.north)
(B.south) edge node[left] {$ g$} (B'.north)
(A.south) edge node[left] {$ p$} (A'.north)
([yshift=-4pt]C'.east) edge node[below] {$\kappa'$} ([yshift=-4pt]A'.west)
([yshift=-4pt]A'.east) edge node[below] {$\alpha'$} ([yshift=-4pt]B'.west)
([yshift=2pt]B'.west) edge node[above] {$e'$} ([yshift=2pt]A'.east)
([yshift=-4pt]C.east) edge node[below] {$\kappa$} ([yshift=-4pt]A.west)
([yshift=-4pt]A.east) edge node[below] {$\alpha$} ([yshift=-4pt]B.west)
([yshift=2pt]B.west) edge node[above] {$e$} ([yshift=2pt]A.east);
\end{tikzpicture} 
\right) =
\begin{tikzpicture}[descr/.style={fill=white},baseline=(O.base)] 
\node (E) at (4,0) {$B \otimes X$};
\node (E') at (4,-2) {$B' \otimes X'$};
\node (F) at (6,0) {$X$};
\node (F') at (6,-2) {$X',$};
\path[->,font=\scriptsize]
(E.east) edge node[above] {$\act$} (F.west)
(E'.east) edge node[above] {$\act'$} (F'.west)
(E.south) edge node[descr] {$ g \otimes v$} (E'.north)
(F.south) edge node[left] {$ v$} (F'.north)
;
\end{tikzpicture} 
  \]
where $\act \colon = \lambda \cdot m \cdot (e \ox \kappa)$. 
The functor $G \colon \mathsf{Act(BiAlg_{\C,coc})} \rightarrow \mathsf{SplitExt(BiAlg_{\C,coc})}$ is defined as
\[ 
G \left( 
\begin{tikzpicture}[descr/.style={fill=white},baseline=(O.base)] 
\node (O) at (-5,-1) {$ $};
\node (E) at (-6.5,0) {$B \otimes X$};
\node (E') at (-6.5,-2) {$B' \otimes X'$};
\node (F) at (-4.5,0) {$X$};
\node (F') at (-4.5,-2) {$X'$};
\path[->,font=\scriptsize]
(E.east) edge node[above] {$\act$} (F.west)
(E'.east) edge node[above] {$\act'$} (F'.west)
(E.south) edge node[descr] {$ g \otimes v$} (E'.north)
(F.south) edge node[left] {$ v$} (F'.north);
\end{tikzpicture} 
 \right) = 
\begin{tikzpicture}[descr/.style={fill=white},baseline=(O.base)] 
\node (O) at (-1,-1) {$ $};
\node (A) at (0,0) {$X \rtimes B$};
\node (B) at (2.5,0) {$B$};
\node (C) at (-2.5,0) {$X$};
\node (A') at (0,-2) {$X' \rtimes B'$};
\node (B') at (2.5,-2) {$B',$};
\node (C') at (-2.5,-2) {$X'$};
\path[dashed,->,font=\scriptsize]
([yshift=2pt]A'.west) edge node[above] {$\pi_1'$} ([yshift=2pt]C'.east)
([yshift=2pt]A.west) edge node[above] {$\pi_1$} ([yshift=2pt]C.east);
\path[->,font=\scriptsize]
(C.south) edge node[left] {$ v$} (C'.north)
(B.south) edge node[left] {$ g$} (B'.north)
(A.south) edge node[left] {$ v \otimes  g$} (A'.north)
([yshift=-4pt]C'.east) edge node[below] {$i_1'$} ([yshift=-4pt]A'.west)
([yshift=-4pt]A'.east) edge node[below] {$\pi_2'$} ([yshift=-4pt]B'.west)
([yshift=2pt]B'.west) edge node[above] {$i_2'$} ([yshift=2pt]A'.east)
([yshift=-4pt]C.east) edge node[below] {$i_1$} ([yshift=-4pt]A.west)
([yshift=-4pt]A.east) edge node[below] {$\pi_2$} ([yshift=-4pt]B.west)
([yshift=2pt]B.west) edge node[above] {$i_2$} ([yshift=2pt]A.east);
\end{tikzpicture} 
  \]
  where $i_1 = 1_X \otimes u_B$, $i_2 = u_X \otimes 1_B $, $\pi_1 = 1_X \otimes \epsilon_B$, $\pi_2= \epsilon_X \otimes 1_B$ and $X \rtimes B$ is the object $X \ox B$ where the bialgebra structure is given by the following morphisms of $\mathcal{C}$
\begin{align*}
m_{X \rtimes B} &= (m \otimes m) \cdot (1_X \otimes \triangleright \otimes 1_B \otimes 1_B) \cdot (1_X \otimes 1_B \otimes \sigma \otimes 1_B) \cdot (1_X \otimes \Delta \otimes 1_X \otimes 1_B)\\
u_{X \rtimes B} &=u_X \otimes u_B,\\
\Delta_{X \rtimes B} &= (1_X \otimes \sigma \otimes 1_B) \cdot (\Delta \otimes \Delta),\\
\epsilon_{X \rtimes B}&= \epsilon_X \otimes \epsilon_B.
\end{align*}
\end{proof}

Let $B$ and $X$ be two cocommutative bialgebras, since the trivial action $\epsilon_B \ox 1_X \colon B \ox X \to X$ is an action of bialgebras as defined in Definition \ref{def action}, we obtain the following corollary.

\begin{corollary}\label{product projection}
Le $B$ and $X$ be two cocommutative bialgebras, the projection $\pi_2$ of the product 
\[ \begin{tikzpicture}[descr/.style={fill=white},baseline=(A.base)] 
\node (A) at (0,0) {$X \ox B$};
\node (B) at (2.5,0) {$B$};
\node (C) at (-2.5,0) {$X$};
\path[->,font=\scriptsize]
([yshift=2pt]A.west) edge node[above] {$\pi_1$} ([yshift=2pt]C.east);
\path[->,font=\scriptsize]
([yshift=-4pt]C.east) edge node[below] {$i_1$} ([yshift=-4pt]A.west)
([yshift=-4pt]A.east) edge node[below] {$\pi_2$} ([yshift=-4pt]B.west)
([yshift=2pt]B.west) edge node[above] {$i_2$} ([yshift=2pt]A.east);
\end{tikzpicture}  \]
where $i_1 = 1_X \otimes u_B$, $i_2 = u_X \otimes 1_B $, $\pi_1 = 1_X \otimes \epsilon_B$ and $\pi_2= \epsilon_X \otimes 1_B$, belongs to $S_{coc}$.
\end{corollary}

\section{$S_{coc}$-protomodularity}
It is known that the category of cocommutative bialgebras is not protomodular. However, we have interesting results with respect to the class $S_{coc}$ of split extensions of cocommutative bialgebras, as the Split Short Five Lemma (Theorem \ref{Sec SSFL}) and the equivalence with the actions (Theorem \ref{equi bialg}). 

These results are the motivation to prove that the category is $S_{coc}$-protomodular, which means protomodular with respect to $S_{coc}$.

Let $\C$ be a pointed category, and $S$ be a class of split epimorphisms with their kernels (also called points), to recall the definition of an $S$-protomodular category we give the definition of a strong point. 

\begin{definition}
A split epimophism wiht its kernel  \begin{tikzpicture}[descr/.style={fill=white},baseline=(A.base),xscale=0.7] 
\node (A) at (0,0) {$A$};
\node (B) at (2.5,0) {$B$};
\node (C) at (-2.5,0) {$X$};
\path[->,font=\scriptsize]
(C.east) edge node[above] {$\kappa$} (A.west)
([yshift=-4pt]A.east) edge node[below] {$\alpha$} ([yshift=-4pt]B.west)
([yshift=2pt]B.west) edge node[above] {$e$} ([yshift=2pt]A.east);
\end{tikzpicture} is called a \emph{strong point} whenever the kernel $\kappa$ and the splitting $e$ are \emph{jointly strongly epimorphic}. 
\end{definition}

\begin{definition}\label{stronglyepic}
The morphisms $v \colon X \to A$ and $w \colon B \to A$ are \emph{jointly strongly epimorphic} if for any morphism $\mu : M \to A$ that factors through $v$ and $w$, $\mu$ is an isomorphism.
\begin{center} \begin{tikzpicture}[descr/.style={fill=white},baseline=(A.base)] 
\node (A) at (0,0) {$A$};
\node (B) at (3,0) {$B$};
\node (C) at (-3,0) {$X$};
\node (D) at (-0,-1.5) {$M$};
\path[->,font=\scriptsize]
(C.east) edge node[above] {$ v $} (A.west)
(C.south east) edge node[left,xshift=-10pt] {$ \delta $} (D.north west)
(B.south west) edge node[right,xshift=10pt] {$\gamma$} (D.north east);
\path[->,font=\scriptsize]
(B.west) edge node[above] {$w$} (A.east);
\path[>->,font=\scriptsize]
(D.north) edge node[right] {$\mu$} (A.south);
\end{tikzpicture}\end{center}
\end{definition}

\begin{definition}\cite{BMS}
Let $\mathcal{C}$ be a pointed finitely complete category and $S$ a class of split epimorphisms stable under pullbacks. $\mathcal{C}$ is said to be \emph{$S$-protomodular} when the class $S$ is closed under finite limits and any point in $S$ is a strong point. 
\end{definition}

\subsection{Pullback stable}

We prove that the class $S_{coc}$ of split extensions of cocommutative bialgebras is pullback stable. 

We consider the pullback of \begin{tikzpicture}[descr/.style={fill=white},baseline=(A.base),xscale=0.7] 
\node (A) at (0,0) {$A$};
\node (B) at (2.5,0) {$B$};
\node (C) at (-2.5,0) {$X$};
\path[dashed,->,font=\scriptsize]
([yshift=2pt]A.west) edge node[above] {$\lambda$} ([yshift=2pt]C.east);
\path[->,font=\scriptsize]
([yshift=-4pt]C.east) edge node[below] {$\kappa$} ([yshift=-4pt]A.west)
([yshift=-4pt]A.east) edge node[below] {$\alpha$} ([yshift=-4pt]B.west)
([yshift=2pt]B.west) edge node[above] {$e$} ([yshift=2pt]A.east);
\end{tikzpicture} along the morphism of bialgebras $g \colon C \to B$.

\begin{equation}\label{pullback}
\begin{tikzpicture}[descr/.style={fill=white},baseline=(current  bounding  box.center),xscale=1.3] 
\node (A) at (0,0) {$A \otimes_B C$};
\node (B) at (2.5,0) {$C$};
\node (C) at (-2.5,0) {$X$};
\node (A') at (0,-2) {$A$};
\node (B') at (2.5,-2) {$B$};
\node (C') at (-2.5,-2) {$X$};
\path[dashed,->,font=\scriptsize]
([yshift=2pt]A'.west) edge node[above] {$\lambda$} ([yshift=2pt]C'.east)
([yshift=2pt]A.west) edge node[above] {$\lambda \cdot p_A$} ([yshift=2pt]C.east)
;
\path[->,font=\scriptsize]
(B) edge node[right] {$ g$}  (B')
 (C.south) edge node[left] {$ 1_X$}  (C'.north)
(A.south) edge node[left] {$ p_A$} (A'.north)
([yshift=-4pt]C'.east) edge node[below] {$\kappa$} ([yshift=-4pt]A'.west)
([yshift=-4pt]A'.east) edge node[below] {$\alpha$} ([yshift=-4pt]B'.west)
([yshift=2pt]B'.west) edge node[above] {$e$} ([yshift=2pt]A'.east)
([yshift=-4pt]C.east) edge node[below] {$(\kappa , u_C \cdot \epsilon_X) $} ([yshift=-4pt]A.west)
([yshift=-4pt]A.east) edge node[below] {$p_C$} ([yshift=-4pt]B.west)
([yshift=2pt]B.west) edge node[above] {$(e \cdot g , 1_C)$} ([yshift=2pt]A.east);
\end{tikzpicture} 
 \end{equation} 
where $(\kappa, u_C \cdot \epsilon_X)$ and $(e \cdot g,1_C)$ are the morphisms induced by the universal property of the pullback. 

\begin{proposition}\label{pullbackstable}
Let \eqref{pullback} be the pullback of \begin{tikzpicture}[descr/.style={fill=white},baseline=(A.base),xscale=0.7] 
\node (A) at (0,0) {$A$};
\node (B) at (2.5,0) {$B$};
\node (C) at (-2.5,0) {$X$};
\path[dashed,->,font=\scriptsize]
([yshift=2pt]A.west) edge node[above] {$\lambda$} ([yshift=2pt]C.east);
\path[->,font=\scriptsize]
([yshift=-4pt]C.east) edge node[below] {$\kappa$} ([yshift=-4pt]A.west)
([yshift=-4pt]A.east) edge node[below] {$\alpha$} ([yshift=-4pt]B.west)
([yshift=2pt]B.west) edge node[above] {$e$} ([yshift=2pt]A.east);
\end{tikzpicture} along a morphism of bialgebras $g \colon C \to B$, then the upper row of \eqref{pullback}:
\begin{equation*}
\begin{tikzpicture}[descr/.style={fill=white},baseline=(current  bounding  box.center),scale=1.3] 
\node (A) at (0,0) {$A \otimes_B C$};
\node (B) at (2.5,0) {$C$};
\node (C) at (-2.5,0) {$X$};
\path[dashed,->,font=\scriptsize]
([yshift=2pt]A.west) edge node[above] {$\lambda \cdot p_A$} ([yshift=2pt]C.east)
;
\path[->,font=\scriptsize]
([yshift=-4pt]C.east) edge node[below] {$(\kappa , u_C \cdot \epsilon_X) $} ([yshift=-4pt]A.west)
([yshift=-4pt]A.east) edge node[below] {$p_C$} ([yshift=-4pt]B.west)
([yshift=2pt]B.west) edge node[above] {$(e \cdot g , 1_C)$} ([yshift=2pt]A.east);
\end{tikzpicture},
 \end{equation*} 
belongs to $S_{coc}$.
\end{proposition}

\begin{proof}
The conditions $(1)$, $(2)$, $(4)$ and $(5)$ of Definition \ref{split extension} are easily checked. We give the details of the condition $(3)$. To show that \[m_{A \otimes_B C} \cdot ((\kappa , u_C \cdot \epsilon_X) \cdot \lambda \cdot p_A  \ox(e \cdot g , 1_C) \cdot p_C ) \cdot \Delta_{A \otimes_B C} = 1_{A \otimes_B C},\] we compose the two sides of the equality with $p_A$ and $p_C$. 

\ctikzfig{piA}

Since $p_A$ and $p_C$ are jointly monic in $\sf BiAlg_{\C,coc}$ (and then also jointly monic in $\sf CoAlg_{\C,coc}$ thanks to the adjunction \eqref{adj}) and the cocommutativity implies that $m_{A \otimes_B C} \cdot ((\kappa , u_C \cdot \epsilon_X) \cdot \lambda \cdot p_A  \ox(e \cdot g , 1_C) \cdot p_C ) \cdot \Delta_{A \otimes_B C}$ is a morphism of coalgebras, we can conclude that the condition $(3)$ holds.
Hence, the class $S_{coc}$ of split extensions of cocommutative bialgebras is stable under pullbacks.
\end{proof}

\subsection{Closure under finite limits}

To prove that the split extensions in $S_{coc}$ are closed under finite limits, we prove that they are closed under products and equalizers. 

\begin{proposition}\label{closedlimit}
The class $S_{coc}$ of split extensions of bialgebras is closed under finite limits. 
\end{proposition}

\begin{proof}
Let \begin{tikzpicture}[descr/.style={fill=white},baseline=(A.base)] 
\node (A) at (0,0) {$A$};
\node (B) at (2.5,0) {$B$};
\node (C) at (-2.5,0) {$X$};
\path[dashed,->,font=\scriptsize]
([yshift=2pt]A.west) edge node[above] {$\lambda$} ([yshift=2pt]C.east);
\path[->,font=\scriptsize]
([yshift=-4pt]C.east) edge node[below] {$\kappa$} ([yshift=-4pt]A.west)
([yshift=-4pt]A.east) edge node[below] {$\alpha$} ([yshift=-4pt]B.west)
([yshift=2pt]B.west) edge node[above] {$e$} ([yshift=2pt]A.east);
\end{tikzpicture}  and \begin{tikzpicture}[descr/.style={fill=white},baseline=(A.base)] 
\node (A) at (0,0) {$A'$};
\node (B) at (2.5,0) {$B'$};
\node (C) at (-2.5,0) {$X'$};
\path[dashed,->,font=\scriptsize]
([yshift=2pt]A.west) edge node[above] {$\lambda'$} ([yshift=2pt]C.east);
\path[->,font=\scriptsize]
([yshift=-4pt]C.east) edge node[below] {$\kappa'$} ([yshift=-4pt]A.west)
([yshift=-4pt]A.east) edge node[below] {$\alpha'$} ([yshift=-4pt]B.west)
([yshift=2pt]B.west) edge node[above] {$e'$} ([yshift=2pt]A.east);
\end{tikzpicture}  two split extensions of bialgebras. We can prove that 
\begin{equation}\label{productdiag}
\begin{tikzpicture}[descr/.style={fill=white},baseline=(A.base),xscale=1.4] 
\node (A) at (0,0) {$A \ox A'$};
\node (B) at (2.5,0) {$B \ox B'$};
\node (C) at (-2.5,0) {$X \ox X'$};
\path[dashed,->,font=\scriptsize]
([yshift=2pt]A.west) edge node[above] {$\lambda \ox \lambda'$} ([yshift=2pt]C.east);
\path[->,font=\scriptsize]
([yshift=-4pt]C.east) edge node[below] {$\kappa \ox \kappa'$} ([yshift=-4pt]A.west)
([yshift=-4pt]A.east) edge node[below] {$\alpha \ox \alpha'$} ([yshift=-4pt]B.west)
([yshift=2pt]B.west) edge node[above] {$e\ox e'$} ([yshift=2pt]A.east);
\end{tikzpicture}  
\end{equation} 
belongs to $S_{coc}$. Since the construction is made component wise it is clear that \eqref{productdiag}  belongs to $S_{coc}$ as we can explicitly see in the proof of condition (3) via the commutativity of the following diagram 
\ctikzfig{product}
The other conditions hold via similar computations

The class $S_{coc}$ is also stable under equalizers. We construct the equalizer of two morphisms of split extensions of bialgebras $(g,v,p)$ and $(g',v',p')$.

\begin{equation*}
\begin{tikzpicture}[descr/.style={fill=white},baseline=(current  bounding  box.center),scale=1.1] 
\node (A) at (0,0) {$A'$};
\node (B) at (2.5,0) {$B'$};
\node (C) at (-2.5,0) {$X'$};
\node (A') at (0,-2) {$A$};
\node (B') at (2.5,-2) {$B$};
\node (C') at (-2.5,-2) {$X$};
\node (A'') at (0,2) {$E$};
\node (B'') at (2.5,2) {$E'$};
\node (C'') at (-2.5,2) {$\hat{E}$};
\path[dashed,->,font=\scriptsize]
([yshift=2pt]A'.west) edge node[above] {$\lambda$} ([yshift=2pt]C'.east)
([yshift=2pt]A.west) edge node[above] {$\lambda'$} ([yshift=2pt]C.east)
([yshift=2pt]A''.west) edge node[above] {$\tilde{\lambda}$} ([yshift=2pt]C''.east)
;
\path[->,font=\scriptsize]
(C'') edge node[left] {$ \hat{\varepsilon}$} (C)
(A'') edge node[left] {$ \varepsilon$} (A)
(B'') edge node[left] {$ \varepsilon'$} (B)
([xshift=-4pt]A.south) edge node[left] {$ g$} ([xshift=-4pt]A'.north)
([xshift=4pt]A.south) edge node[right] {$ g'$} ([xshift=4pt]A'.north)
([xshift=-4pt]C.south) edge node[left] {$ v$} ([xshift=-4pt]C'.north)
([xshift=4pt]C.south) edge node[right] {$ v'$} ([xshift=4pt]C'.north)
([xshift=-4pt]B.south) edge node[left] {$ p$} ([xshift=-4pt]B'.north)
([xshift=4pt]B.south) edge node[right] {$ p'$} ([xshift=4pt]B'.north)
([yshift=-4pt]C''.east) edge node[below] {$\tilde{\kappa}$} ([yshift=-4pt]A''.west)
([yshift=-4pt]A''.east) edge node[below] {$\tilde{\alpha}$} ([yshift=-4pt]B''.west)
([yshift=2pt]B''.west) edge node[above] {$\tilde{e}$} ([yshift=2pt]A''.east)
([yshift=-4pt]C'.east) edge node[below] {$\kappa$} ([yshift=-4pt]A'.west)
([yshift=-4pt]A'.east) edge node[below] {$\alpha$} ([yshift=-4pt]B'.west)
([yshift=2pt]B'.west) edge node[above] {$e$} ([yshift=2pt]A'.east)
([yshift=-4pt]C.east) edge node[below] {$\kappa'$} ([yshift=-4pt]A.west)
([yshift=-4pt]A.east) edge node[below] {$\alpha'$} ([yshift=-4pt]B.west)
([yshift=2pt]B.west) edge node[above] {$e'$} ([yshift=2pt]A.east);
\end{tikzpicture} 
\end{equation*}
where $\tilde{\alpha}$, $\tilde{e}$, $\tilde{\kappa}$ and $\tilde{\lambda}$ are induced by the universal properties of the equalizers. Note that $\tilde{\lambda}$ is induced by the universal property of $\hat{\varepsilon}$ seen as a coequalizer in $\sf CoAlg_{\C,coc}$ via \eqref{adj}. 
By using the fact that $\varepsilon, \varepsilon'$ and $\hat{\varepsilon}$ are monomorphisms of bialgebras (and hence also of coalgebras) we can conclude that 
\begin{equation}\label{equalizer}
\begin{tikzpicture}[descr/.style={fill=white},baseline=(current  bounding  box.center),scale=1.3] 
\node (A'') at (0,2) {$E$};
\node (B'') at (2.5,2) {$E'$};
\node (C'') at (-2.5,2) {$\hat{E}$};
\path[dashed,->,font=\scriptsize]
([yshift=2pt]A''.west) edge node[above] {$\tilde{\lambda}$} ([yshift=2pt]C''.east)
;
\path[->,font=\scriptsize]
([yshift=2pt]B''.west) edge node[above] {$\tilde{e}$} ([yshift=2pt]A''.east)
([yshift=-4pt]C''.east) edge node[below] {$\tilde{\kappa}$} ([yshift=-4pt]A''.west)
([yshift=-4pt]A''.east) edge node[below] {$\tilde{\alpha}$} ([yshift=-4pt]B''.west);
\end{tikzpicture} 
\end{equation}
belongs to $S_{coc}$. We give an explicit proof of the condition (4) via the commutativity of this diagram:
\ctikzfig{equalizer}
Since $\hat{\varepsilon}$ is a monomorphism of coalgebras and $\tilde{\lambda}\cdot m \cdot (\tilde{\kappa} \otimes \tilde{e})$ is a coalgebra morphism, we can conclude that \eqref{equalizer} satisfies condition (4).
\end{proof}

\subsection{Strong points}

It was proven in \cite{Stercksplit} that for a split extension of (cocommutative) bialgebras 
\[ \begin{tikzpicture}[descr/.style={fill=white},baseline=(A.base)] 
\node (A) at (0,0) {$A$};
\node (B) at (2.5,0) {$B$};
\node (C) at (-2.5,0) {$X$};
\path[dashed,->,font=\scriptsize]
([yshift=2pt]A.west) edge node[above] {$\lambda$} ([yshift=2pt]C.east);
\path[->,font=\scriptsize]
([yshift=-4pt]C.east) edge node[below] {$\kappa$} ([yshift=-4pt]A.west)
([yshift=-4pt]A.east) edge node[below] {$\alpha$} ([yshift=-4pt]B.west)
([yshift=2pt]B.west) edge node[above] {$e$} ([yshift=2pt]A.east);
\end{tikzpicture}  \]
$\kappa$ and $e$ are jointly epimorphic. Now we prove that for cocommutative biagebras, $\kappa$ and $e$ are jointly strongly epimorphic and hence any split extension of cocommutative bialgebras is a strong point. 

\begin{proposition}\label{strong}
Any split extension of cocommutative bialgebras is a strong point. 
\end{proposition}

\begin{proof}
Let \begin{tikzpicture}[descr/.style={fill=white},baseline=(A.base)] 
\node (A) at (0,0) {$A$};
\node (B) at (2.5,0) {$B$};
\node (C) at (-2.5,0) {$X$};
\path[dashed,->,font=\scriptsize]
([yshift=2pt]A.west) edge node[above] {$\lambda$} ([yshift=2pt]C.east);
\path[->,font=\scriptsize]
([yshift=-4pt]C.east) edge node[below] {$\kappa$} ([yshift=-4pt]A.west)
([yshift=-4pt]A.east) edge node[below] {$\alpha$} ([yshift=-4pt]B.west)
([yshift=2pt]B.west) edge node[above] {$e$} ([yshift=2pt]A.east);
\end{tikzpicture} be a split extension of cocommutative bialgebras and $\mu \colon M \to A$ be a monomorphism of bialgebras. Let $\kappa$ and $e$ factor through $\mu$: 

\begin{center} \begin{tikzpicture}[descr/.style={fill=white},baseline=(A.base)] 
\node (A) at (0,0) {$A$};
\node (B) at (3,0) {$B$};
\node (C) at (-3,0) {$X$};
\node (D) at (-0,-1.5) {$M$};
\path[->,font=\scriptsize]
(C.east) edge node[above] {$ \kappa $} (A.west)
(C.south east) edge node[left,xshift=-10pt] {$ \delta $} (D.north west)
(B.south west) edge node[right,xshift=10pt] {$\gamma$} (D.north east);
\path[->,font=\scriptsize]
(B.west) edge node[above] {$e$} (A.east);
\path[>->,font=\scriptsize]
(D.north) edge node[right] {$\mu$} (A.south);
\end{tikzpicture}\end{center}

We can form the following commutative diagram 

\begin{equation*}
\begin{tikzpicture}[descr/.style={fill=white},baseline=(current  bounding  box.center),scale=1.1] 
\node (A) at (0,0) {$M$};
\node (B) at (2.5,0) {$B$};
\node (C) at (-2.5,0) {$X$};
\node (A') at (0,-2) {$A$};
\node (B') at (2.5,-2) {$B$};
\node (C') at (-2.5,-2) {$X$};
\path[dashed,->,font=\scriptsize]
([yshift=2pt]A'.west) edge node[above] {\normalsize $\lambda$} ([yshift=2pt]C'.east)
([yshift=2pt]A.west) edge node[above] {$\lambda \cdot \mu$} ([yshift=2pt]C.east)
;
\path[->,font=\scriptsize]
(B) edge node[right] {$ 1_B$}  (B')
 (C.south) edge node[left] {$ 1_X$}  (C'.north)
(A.south) edge node[left] {$ \mu$} (A'.north)
([yshift=-4pt]C'.east) edge node[below] {$\kappa$} ([yshift=-4pt]A'.west)
([yshift=-4pt]A'.east) edge node[below] {$\alpha$} ([yshift=-4pt]B'.west)
([yshift=2pt]B'.west) edge node[above] {$e$} ([yshift=2pt]A'.east)
([yshift=-4pt]C.east) edge node[below] {$\delta$} ([yshift=-4pt]A.west)
([yshift=-4pt]A.east) edge node[below] {$\alpha \cdot \mu$} ([yshift=-4pt]B.west)
([yshift=2pt]B.west) edge node[above] {$\gamma$} ([yshift=2pt]A.east);
\end{tikzpicture} 
 \end{equation*} 
 
 We prove that the upper row is in the class $S_{coc}$ by checking all the conditions of Definition \ref{split extension}. In particular, via the commutativity of the following diagram
 \ctikzfig{mu}
 
%
we conclude that \begin{tikzpicture}[descr/.style={fill=white},baseline=(current  bounding  box.center),scale=1.1] 
\node (A) at (0,0) {$M$};
\node (B) at (2.5,0) {$B$};
\node (C) at (-2.5,0) {$X$};
\path[dashed,->,font=\scriptsize]
([yshift=2pt]A.west) edge node[above] {$\lambda \cdot \mu$} ([yshift=2pt]C.east)
;
\path[->,font=\scriptsize]
([yshift=-4pt]C.east) edge node[below] {$\delta$} ([yshift=-4pt]A.west)
([yshift=-4pt]A.east) edge node[below] {$\alpha \cdot \mu$} ([yshift=-4pt]B.west)
([yshift=2pt]B.west) edge node[above] {$\gamma$} ([yshift=2pt]A.east);
\end{tikzpicture} satisfies condition $(3)$ since $\mu$ is a monomorphism of bialgebras (and then also a monomorphism of coalgebras \eqref{adj}) and $m \cdot ( \delta \cdot \lambda \cdot \mu \ox \gamma \cdot \alpha \cdot \mu  ) \cdot \Delta$ is a coalgebra morphism thanks to the cocommutativity.
\end{proof}

Thanks to the previous results we obtain the following theorem:

\begin{theorem}\label{thmproto}
The category of cocommutative bialgebras in any symmetric monoidal category is an $S_{coc}$-protomodular category with respect to the class $S_{coc}$ of split extensions of cocommutative bialgebras defined in Definition \ref{split extension}.
\end{theorem}
\begin{proof}
This theorem holds thanks to Proposition \ref{pullbackstable}, Proposition \ref{closedlimit} and Proposition \ref{strong}.
\end{proof}

\begin{remark}
This theorem generalizes the result of \cite{BMS} saying that the category of monoids is $S$-protomodular with respect to the class of Schreier split epimorphisms. Indeed, if $\C$ is the symmetric monoidal category $\sf (Set, \times, \{ \star \})$, then $S_{coc}$ becomes the class of the Schreier split epimorphisms and $\sf BiAlg_{Set,coc}$ the category of monoids.
\end{remark}
\section{Huq commutator}

In this section, we recall the notion of centrality in the sense of Huq in a pointed category with binary products. Moreover, we give an explicit description of this centrality for two subbialgebras of a cocommutative bialgebra.

Let $\mathcal{C}$ be a pointed category with binary products, we say that two subobjects $x \colon X \to A$ and $y \colon Y \to A$ commute (or centralize) in the sense of Huq \cite{Huq} if there exists a morphism $p \colon X \times Y \to A$ such that the following diagram commutes
\begin{equation*}
\begin{tikzpicture}[descr/.style={fill=white},yscale=0.9,xscale=1.1]
\node (A) at (0,0) {$A$};
\node (C) at (2,2) {$Y$};
\node (D) at (0,2) {$X \times Y$};
\node (B) at (-2,2) {$X$};
  \path[-stealth]
 (B.south) edge node[left] {${x\,} $} (A.north west) 
 (C.south) edge node[right] {${\, y} $} (A.north east) 
 (B.east) edge node[above] {$(1,0)$} (D.west) 
 (C.west) edge node[above] {$(0,1)$} (D.east);
   \path[-stealth,dashed]
 (D.south) edge node[right] {$p$} (A.north);
\end{tikzpicture}
\end{equation*}
Then we denote by $[X,Y] = 0$ the fact that the subobjects $X$ and $Y$ commute.

Note that in $\sf BiAlg_{\C,coc}$, if it exists, $p$ is unique thanks to Corollary \ref{product projection} and Proposition \ref{strong}.
\begin{proposition}
In $\sf BiAlg_{\C,coc}$, the following are equivalent for $x \colon X \to A$ and $y \colon Y \to A$ two subobjects of a cocommutative bialgebra $A$:
\begin{itemize}
\item[(i)] $m \cdot \sigma \cdot (x \ox y) = m  \cdot (x \ox y)  $
\item[(ii)] $[X,Y]$ = 0, i.e\ there exists a (unique) morphism of bialgebras $p \colon X \otimes Y \to A$ such that the following diagram commutes
\begin{equation}\label{connector}
\begin{tikzpicture}[descr/.style={fill=white},yscale=0.9,xscale=1]
\node (A) at (0,0) {$A$};
\node (C) at (2.5,2) {$Y$};
\node (D) at (0,2) {$X \otimes Y$};
\node (B) at (-2.5,2) {$X$};
  \path[-stealth]
 (B.south) edge node[left] {${x\,} $} (A.north west) 
 (C.south) edge node[right] {${\, y} $} (A.north east) 
 (B.east) edge node[above] {$1_X \ox u_Y$} (D.west) 
 (C.west) edge node[above] {$u_X \otimes 1_Y$} (D.east);
   \path[-stealth,dashed]
 (D.south) edge node[right] {$p$} (A.north);
\end{tikzpicture}
\end{equation}
\end{itemize}
\end{proposition}
\begin{proof}
$(i) \Rightarrow (ii)$, we define $p : X \ox Y \to A$ by $p \coloneqq m \cdot (x \ox y)$. It is easy to see that this map makes the diagram \eqref{connector} commutes. Moreover, thanks to $(i)$ we can prove that $p$ is a morphism of bialgebras (this proof is straghforward and is left to the reader). 
%
On the other way around, since $p$ is a morphism of bialgebras which makes \eqref{connector} commute, we can make the following diagram commute
\ctikzfig{huq}
and conclude that $(i)$ holds.
\end{proof}

Note that we obtained a similar result in the case of cocommutative Hopf algebras in the paper \cite{GSV}.

\section{Smith is Huq}

Now we would like to compare the notion of centrality in the sense of Huq and in the sense of Smith. 

We first recall the centrality of two equivalence relations in the sense of Smith in any category with pullbacks \cite{BG}. 

\begin{definition}\label{connectorSmith}
Let $(R,r_0,r_1)$ and $(S,s_0,s_1)$  be two equivalence relations over the same object $X$. We denote the pullback of $s_0$ along $r_1$ by 
  \begin{center}
\begin{tikzpicture}[descr/.style={fill=white},yscale=1.2]
\node (A) at (0,0) {$R$};
\node (B) at (0,1.5) {$R \times_X S$};
\node (C) at (3,1.5) {$S$};
\node (D) at (3,0) {$X.$};
 \path[->,font=\scriptsize]
 (C.south) edge node[right] {$s_0$} (D.north)
  (B.south) edge node[right] {$p_1$} (A.north)
    (B) edge node[above] {$p_2$} (C)
       (A) edge node[above] {$r_1$} (D)
;
\end{tikzpicture}
\end{center}
 A connector between $R$ and $S$ is an arrow $\hat{p} \colon R \times_X S \to X$ such that 
 \begin{itemize}
 \item $xS\hat{p}(x,y,z)$ and $zR\hat{p}(x,y,z)$,
 \item $\hat{p}(x,x,y) = y$ and $\hat{p}(x,y,y)=x$,
 \item $\hat{p}(x,y,\hat{p}(y,u,v)) = \hat{p}(x,u,v) $ and $\hat{p}(\hat{p}(x,y,u),u,v) = \hat{p}(x,y,v)$,

 \end{itemize}
 where the above expressions are defined.
  When such an arrow exists, we say that the relations $R$ and $S$ centralize (in the sense of Smith).
\end{definition}

Note that this notion of centrality of equivalence relations is not independent of the centrality in the sense of Huq. If two equivalence relations centralize each other (in the sense of Smith) then it is true that the normal subobjects associated with them, called normalizations \cite{Bourn}, centralize in the sense of Huq \cite{BG}.


 The converse is not true in general. If $\mathcal{C}$ is a category such that any two equivalence relations always centralize each other as soon as their normalizations centralize in the sense of Huq, will say that $\C$ satisfies the \textit{Smith is Huq} property \cite{BG}.

For example, the categories of groups and cocommutative Hopf algebras over a field satisfy this condition \cite{GSV}. 

In the paper \cite{MFM}, the authors studied this condition in the context of $S$-protomodular categories. 

They considered pointed $S$-protomodular categories with respect to a class of points $S$ that are stable under composition and such that any product projection, i.e. any such diagram
\[\begin{tikzpicture}[descr/.style={fill=white},baseline=(A.base)] 
\node (A) at (0,0) {$X \x B$};
\node (B) at (2.5,0) {$B,$};
\node (C) at (-2.5,0) {$X$};
\path[->,font=\scriptsize]
([yshift=2pt]A.west) edge node[above] {$\pi_1$} ([yshift=2pt]C.east);
\path[->,font=\scriptsize]
([yshift=-4pt]C.east) edge node[below] {$i_1$} ([yshift=-4pt]A.west)
([yshift=-4pt]A.east) edge node[below] {$\pi_2$} ([yshift=-4pt]B.west)
([yshift=2pt]B.west) edge node[above] {$i_2$} ([yshift=2pt]A.east);
\end{tikzpicture}  \]
where $i_1 \coloneqq (1_X,0)$ and $i_2 \coloneqq (0,1_B)$, belongs to the class $S$.

In that context, the condition \emph{Smith is Huq} means that two $S$-equivalence relations centralize each other (in the sense of Smith) if and only if their normalization commute in the sense of Huq. The category of monoids (but also the category of monoids with operations \cite{MFMS}) satisfies this property, with respect to the class of Schreier split epimorphisms.

First, we prove that the class $S_{coc}$ is stable under composition.
\begin{proposition}
The class $S_{coc}$ of split extensions of cocommutative bialgebras is closed under composition. 
\end{proposition}

\begin{proof}
Let 
\begin{tikzpicture}[descr/.style={fill=white},baseline=(A.base)] 
\node (A) at (0,0) {$A$};
\node (B) at (2.5,0) {$B$};
\node (C) at (-2.5,0) {$X$};
\path[dashed,->,font=\scriptsize]
([yshift=2pt]A.west) edge node[above] {$\lambda$} ([yshift=2pt]C.east);
\path[->,font=\scriptsize]
([yshift=-4pt]C.east) edge node[below] {$\kappa$} ([yshift=-4pt]A.west)
([yshift=-4pt]A.east) edge node[below] {$\alpha$} ([yshift=-4pt]B.west)
([yshift=2pt]B.west) edge node[above] {$e$} ([yshift=2pt]A.east);
\end{tikzpicture} 
and
\begin{tikzpicture}[descr/.style={fill=white},baseline=(A.base)] 
\node (A) at (0,0) {$B$};
\node (B) at (2.5,0) {$C$};
\node (C) at (-2.5,0) {$Y$};
\path[dashed,->,font=\scriptsize]
([yshift=2pt]A.west) edge node[above] {$\lambda'$} ([yshift=2pt]C.east);
\path[->,font=\scriptsize]
([yshift=-4pt]C.east) edge node[below] {$\kappa'$} ([yshift=-4pt]A.west)
([yshift=-4pt]A.east) edge node[below] {$\alpha'$} ([yshift=-4pt]B.west)
([yshift=2pt]B.west) edge node[above] {$e'$} ([yshift=2pt]A.east);
\end{tikzpicture} be two composable split extensions. We can build the following diagram
\[
\begin{tikzpicture}[descr/.style={fill=white},baseline=(A.base)] 
\node (A) at (0,0) {$A$};
\node (B) at (2.5,0) {$C$};
\node (C) at (-2.5,0) {$Z$};
\path[dashed,->,font=\scriptsize]
([yshift=2pt]A.west) edge node[above] {$\hat{\lambda}$} ([yshift=2pt]C.east);
\path[->,font=\scriptsize]
([yshift=-4pt]C.east) edge node[below] {$\hat{\kappa}$} ([yshift=-4pt]A.west)
([yshift=-4pt]A.east) edge node[below] {$\alpha' \cdot \alpha$} ([yshift=-4pt]B.west)
([yshift=2pt]B.west) edge node[above] {$e \cdot e'$} ([yshift=2pt]A.east);
\end{tikzpicture} 
\]
where $\hat{\kappa} \colon Z \to A$ is the kernel of $\alpha' \cdot \alpha$ and $\hat{\lambda}$ is the factorization through $\hat{\kappa}$ of the morphism of coalgebras $ m \cdot (\kappa\cdot \lambda \ox e \cdot \kappa' \cdot \lambda' \cdot \alpha) \cdot \Delta $. 
In particular, we have the following equality  
\begin{equation}\label{kernelaa'}
 \hat{\kappa} \cdot \hat{\lambda} = m \cdot (\kappa\cdot \lambda \ox e \cdot \kappa' \cdot \lambda' \cdot \alpha) \cdot \Delta . 
\end{equation} 
We prove that this construction belongs to $S_{coc}$.
 The condition $(3)$ is proven via the commutativity of the following diagram.

  \begin{tikzpicture}[font=\small,xscale=1.4]
	\begin{pgfonlayer}{nodelayer}
		\node [style=green] (2) at (-14.25, 6.75) {$A$};
		\node [style=green] (4) at (-11, 6.75) {$A^2$};
		\node [style=green] (5) at (4, 6.75) {$A^2$};
		\node [style=green] (7) at (6.25, 6.75) {$A$};
		\node [style=none] (12) at (-12.75, 7.25) {$\Delta$};
		\node [style=none] (13) at (-6.25, 7.25) {$\hat{\kappa} \cdot \hat{\lambda} \otimes 1_A$};
		\node [style=none] (15) at (5.25, 7.25) {$m$};
		\node [style=green] (17) at (-2.25, 6.75) {$A^2$};
		\node [style=none] (19) at (0.75, 7.25) {$1_A \otimes e \cdot e' \cdot \alpha' \cdot \alpha$};
		\node [style=green] (21) at (-9.75, 3.75) {$A^3$};
		\node [style=green] (22) at (-2.25, 3.75) {$A^3$};
		\node [style=none] (23) at (-9, 5) {$\Delta \otimes 1_A$};
		\node [style=none] (24) at (-5.75, 4.25) {$\kappa\cdot \lambda \otimes e \cdot \kappa' \cdot \lambda' \cdot \alpha \otimes 1_A$};
		\node [style=none] (25) at (-1, 5) {$m \otimes 1_A$};
		\node [style=green] (26) at (-13, 3.75) {$A^2$};
		\node [style=none] (27) at (-13, 5.25) {$\Delta$};
		\node [style=none] (28) at (-11.75, 4.25) {$1_A \otimes \Delta$};
		\node [style=green] (29) at (2.25, 3.75) {$A^2$};
		\node [style=green] (32) at (4.5, 0.5) {$A^2$};
		\node [style=green] (33) at (3, 2.25) {$1_A \otimes m$};
		\node [style=green] (34) at (5.75, 4) {$m$};
		\node [style=green] (36) at (0.25, 0.5) {$A \otimes B$};
		\node [style=green] (38) at (-3.5, 0.5) {$A \otimes B^2$};
		\node [style=none] (39) at (-1.75, 1) {$1_A \otimes m$};
		\node [style=green] (42) at (-9.75, 0.5) {$A \otimes B^2$};
		\node [style=green] (43) at (-13, 0.5) {$A\otimes B$};
		\node [style=none] (44) at (-11.5, 1) {$1_A \otimes \Delta$};
		\node [style=green] (45) at (-12.75, 2.25) {$1_A \otimes \alpha$};
		\node [style=none] (46) at (-7, 0) {$\kappa \cdot \lambda \otimes \kappa' \cdot \lambda' \otimes e' \cdot \alpha'$};
		\node [style=none] (47) at (-9.25, -2) {$\kappa \cdot \lambda \otimes 1_A$};
		\node [style=none] (48) at (-14.25, -3.25) {};
		\node [style=green] (49) at (-14.25, -3.25) {$A$};
		\node [style=green] (50) at (6.25, -3.25) {$A$};
		\node [style=none] (52) at (-4.75, -2.75) {$1_A$};
		\node [style=none] (53) at (-14.75, 2) {$1_A$};
		\node [style=none] (55) at (6.75, 1.75) {$1_A$};
		\node [style=green] (56) at (3.75, 5.25) {$m \otimes 1_A$};
		\node [style=none] (58) at (0, 3.25) {$1_A^2 \otimes e \cdot e' \cdot \alpha' \cdot \alpha$};
		\node [style=none] (59) at (2.25, 0) {$1_A \otimes e$};
		\node [style=none] (63) at (0, -1.25) {$(3)$};
		\node [style=none] (64) at (-12, 5.75) {$\eqref{coass comultiplication}$};
		\node [style=none] (65) at (-6, 5.5) {$\eqref{kernelaa'}$};
		\node [style=none] (66) at (4, 3.5) {$\eqref{ass}$};
		\node [style=none] (67) at (-6.5, -1) {$(3)$};
	\end{pgfonlayer}
	\begin{pgfonlayer}{edgelayer}
		\draw [style=stealth] (2) to (4);
		\draw [style=stealth] (5) to (7);
		\draw [style=stealth] (17) to (5);
		\draw [style=stealth] (4) to (17);
		\draw [style=stealth] (21) to (22);
		\draw [style=stealth] (22) to (17);
		\draw [style=stealth] (4) to (21);
		\draw [style=stealth] (26) to (21);
		\draw [style=stealth] (2) to (26);
		\draw [style=stealth] (22) to (29);
		\draw [style=stealth] (29) to (32);
		\draw [style=stealth] (32) to (7);
		\draw [style=stealth] (36) to (32);
		\draw [style=stealth] (38) to (36);
		\draw [style=stealth] (26) to (43);
		\draw [style=stealth] (43) to (42);
		\draw [style=stealth] (42) to (38);
		\draw [style=stealth, bend right] (43) to (36);
		\draw [style=stealth] (2) to (49);
		\draw [style=stealth] (49) to (50);
		\draw [style=stealth] (50) to (7);
		\draw [style=stealth] (29) to (5);
	\end{pgfonlayer}
\end{tikzpicture}

The condition $(4)$ holds thanks to the commutativity of the two diagrams in Figure \ref{(4)composition},
where $\overline{\alpha \cdot \hat{\kappa}}$ is the factorization of $\alpha \cdot \hat{\kappa}$ through the kernel $\kappa'$ of $\alpha'$:
\begin{equation}\label{factorise kernel}
\kappa' \cdot \overline{\alpha \cdot \hat{\kappa}} = \alpha \cdot \hat{\kappa},
\end{equation}
and the fact that $\hat{\lambda} \cdot m \cdot (\hat{\kappa} \ox e \cdot e') $ and $( 1_Z \ox \epsilon_C)  $ are coalgebras morphisms and $\hat{\kappa}$ is a monomorphism of coalgebras.  
\end{proof}

Thanks to this proposition and Corollary \ref{product projection}, $\sf BiAlg_{\C,coc}$ is a category in which we can apply the results obtained in \cite{MFM}. To apply the results of \cite{MFM}, we recall the notions of $S$-equivalence relation and reflexive-multiplicative graph. 

\begin{definition}
An equivalence relation \[ 
\begin{tikzpicture}[descr/.style={fill=white},baseline=(A.base)] 
\node (A) at (0,0) {$A$};
\node (B) at (2.5,0) {$B$};
\path[->,font=\scriptsize]
([yshift=0pt]B.west) edge node[descr] {$e$} ([yshift=0pt]A.east)
([yshift=-6pt]A.east) edge node[below] {$\beta$} ([yshift=-6pt]B.west)
([yshift=6pt]A.east) edge node[above] {$\alpha$} ([yshift=6pt]B.west);
\end{tikzpicture}\] is an $S$-equivalence relation if the point $(\alpha, e)$ is in $S$.
\end{definition}
\begin{definition}\cite{CPP}
In a category $\mathcal{C}$ with pullbacks, a reflexive graph, denoted by 
 \begin{tikzpicture}[descr/.style={fill=white},xscale=0.8,yscale=1,baseline=(A.base)]
\node (A) at (0,0) {$A_1$};
\node (D) at (3,0) {$A_0$};
 \path[->,font=\scriptsize]
([yshift=7pt]A.east) edge node[above] {$\delta$} ([yshift=7pt]D.west)
(D.west) edge node[descr] {$\iota$} (A.east)
([yshift=-7pt]A.east) edge node[below] {$\gamma$} ([yshift=-7pt]D.west);
\end{tikzpicture}, 
together with a morphism $c \colon A_1 \times_{A_0} A_1 \rightarrow A_1$ is called a \emph{reflexive-multiplicative graph} 
\begin{equation}\label{mgraph}
 \begin{tikzpicture}[descr/.style={fill=white},yscale=1.2,baseline=(A.base)]
\node (A) at (0,0) {$A_1$};
\node (D) at (3,0) {$A_0$};
\node (X) at (-3,0) {$A_1\times_{A_0}A_1$};
 \path[->,font=\scriptsize]
 (X.east) edge node[above] {$c$} (A.west)
([yshift=7pt]A.east) edge node[above] {$\delta$} ([yshift=7pt]D.west)
(D.west) edge node[descr] {$\iota$} (A.east)
([yshift=-7pt]A.east) edge node[below] {$\gamma$} ([yshift=-7pt]D.west);
\end{tikzpicture} ,
\end{equation}
where  $A_1\times_{A_0}A_1$ is the (object part of the) following pullback
\[
\xymatrix{A_1\times_{A_0} A_1 \ar[d]_{p_1} \ar[r]^-{p_2} & A_1 \ar[d]^{ \gamma} \\
A_1 \ar[r]_-{\delta}& A_0
}
\] 
and $c$ is a multiplication that is required to satisfy the identities
\begin{equation}\label{RGM}
c \cdot (1_{A_1} , \iota \cdot \delta ) = 1_{A_1} = c \cdot (\iota \cdot \gamma, 1_{A_1} ),
\end{equation} 
 where $(1_{A_1} , \iota \cdot \delta ) \colon A_1 \rightarrow A_1\times_{A_0}A_1 $ and $(\iota \cdot \gamma, 1_{A_1} ) \colon A_1 \rightarrow A_1\times_{A_0}A_1 $ are induced by the universal property of the pullback $A_1\times_{A_0}A_1$.
 \end{definition}

To apply Theorem 4.2 in \cite{MFM} we need the following result.

\begin{proposition}
Every reflexive graph \[ 
\begin{tikzpicture}[descr/.style={fill=white},baseline=(A.base)] 
\node (A) at (0,0) {$A$};
\node (B) at (2.5,0) {$B$};
\node (C) at (-2.5,0) {$X$};
\node (C') at (0,2) {$X'$};
\path[->,dashed,font=\scriptsize]

([yshift=2pt]A.west) edge node[above] {$\lambda$} ([yshift=2pt]C.east)
([xshift=2pt]A.north) edge node[right] {$\lambda'$} ([xshift=2pt]C'.south);
\path[->,font=\scriptsize]

([yshift=0pt]B.west) edge node[descr] {$e$} ([yshift=0pt]A.east)
([yshift=-6pt]A.east) edge node[below] {$\beta$} ([yshift=-6pt]B.west)
([xshift=-4pt]C'.south) edge node[left] {$\kappa'$} ([xshift=-4pt]A.north)
([yshift=-4pt]C.east) edge node[below] {$\kappa$} ([yshift=-4pt]A.west)
([yshift=6pt]A.east) edge node[above] {$\alpha$} ([yshift=6pt]B.west);
\end{tikzpicture}\]
such that \begin{tikzpicture}[descr/.style={fill=white},baseline=(A.base)] 
\node (A) at (0,0) {$A$};
\node (B) at (2.5,0) {$B$};
\node (C) at (-2.5,0) {$X$};
\path[dashed,->,font=\scriptsize]
([yshift=2pt]A.west) edge node[above] {$\lambda$} ([yshift=2pt]C.east);
\path[->,font=\scriptsize]
([yshift=-4pt]C.east) edge node[below] {$\kappa$} ([yshift=-4pt]A.west)
([yshift=-4pt]A.east) edge node[below] {$\alpha$} ([yshift=-4pt]B.west)
([yshift=2pt]B.west) edge node[above] {$e$} ([yshift=2pt]A.east);
\end{tikzpicture} 
and
\begin{tikzpicture}[descr/.style={fill=white},baseline=(A.base)] 
\node (A) at (0,0) {$A$};
\node (B) at (2.5,0) {$B$};
\node (C) at (-2.5,0) {$X'$};
\path[dashed,->,font=\scriptsize]
([yshift=2pt]A.west) edge node[above] {$\lambda'$} ([yshift=2pt]C.east);
\path[->,font=\scriptsize]
([yshift=-4pt]C.east) edge node[below] {$\kappa'$} ([yshift=-4pt]A.west)
([yshift=-4pt]A.east) edge node[below] {$\beta$} ([yshift=-4pt]B.west)
([yshift=2pt]B.west) edge node[above] {$e$} ([yshift=2pt]A.east);
\end{tikzpicture} belong to $S_{coc}$ and $[ X , X'] = 0 $, is a reflexive-multiplicative graph.
\end{proposition}

\begin{proof}
Let consider the following pullback
\[
\begin{tikzpicture}[descr/.style={fill=white},baseline=(current  bounding  box.center),yscale=0.9,xscale=1.1] 
\node (A) at (0,0) {$A \otimes_B A$};
\node (B) at (2.5,0) {$A$};
\node (A') at (0,-2) {$A$};
\node (B') at (2.5,-2) {$B$};
\path[->,font=\scriptsize]
(B) edge node[right] {$\beta$}  (B')
(A.south) edge node[left] {$ p_1$} (A'.north)
(A'.east) edge node[below] {$\alpha$} (B'.west)
(A.east) edge node[below] {$p_2$} (B.west)
;
\end{tikzpicture} \]
where $A \otimes_B A$ is the object in the following equalizer
\begin{equation}\label{eq compo}
\begin{tikzpicture}[descr/.style={fill=white},baseline=(A.base),xscale=1.5] 
\node (A) at (-1,0) {$A \otimes A$};
\node (B) at (2.5,0) {$A \ox B \ox A $};
\node (C) at (-2.5,0) {$A \otimes_B A$};
\path[->,font=\scriptsize]
([yshift=-4pt]A.east) edge node[below] {$(1_A \ox \alpha \ox 1_A) \cdot ( \Delta \ox 1_A) $} ([yshift=-4pt]B.west)
([yshift=0pt]C.east) edge node[above] {$\varepsilon$} ([yshift=0pt]A.west)
([yshift=4pt]A.east) edge node[above] {$(1_A \otimes \beta \ox 1_C) \cdot (1_A \otimes \Delta)  $} ([yshift=4pt]B.west)
;
\end{tikzpicture}.
\end{equation}

By defining $c \colon A \ox_B A \to A$ by 
\[c =  m \cdot (\kappa \cdot \lambda \ox 1_A) \cdot (\varepsilon \otimes \varepsilon)\]
we can verify that this reflexive graph is multiplicative thanks to Figure \ref{cmorph1} 
and Figure \ref{cmorph2}.
\end{proof}

This result and Theorem 4.2 in \cite{MFM}, implies the following result

\begin{theorem}\label{thmhuq}
In $\sf BiAlg_{\C,coc}$, let $S_{coc}$ be the class of split extensions of cocommutative bialgebras as defined in Definition \ref{split extension}, two $S_{coc}$-equivalence relations centralize each other (in the sense of Smith) if and only if their normalization commute in the sense of Huq. 
\end{theorem}

By restricting the above theorem to the symmetric monoidal category $\sf (Set, \times , \{ \star \})$, we obtain the result on monoids of \cite{MFM}. Note that, in \cite{MFM}, they also prove this result for every category of monoids with operations.

\section{Conclusion}
In \cite{Stercksplit}, we introduced a class of split extensions of (cocommutative) bialgebras, called $S_{coc}$, such that their category is equivalent to the category of actions of (cocommutative) bialgebras and we proved the Split Short Five Lemma when we restrict it to the split extensions of (cocommutative) bialgebras. In this paper, we use this class $S_{coc}$, to prove that $\sf BiAlg_{\C,coc}$ is $S_{coc}$-protomodular. It implies that we can now apply the theory and results of \cite{BM3x3,Bournpartial} to obtain new results for cocommutative bialgebras. For example, it follows that $\sf BiAlg_{\C,coc}$ is a $S_{coc}$-Mal'tsev category \cite{BM3x3,Bournpartial}, hence any $S_{coc}$-reflexive relation is transitive.
Another result of this paper is the description of the notion of centrality in the sense of Huq for cocommutative bialgebras. Moreover, we prove that  $\sf BiAlg_{\C,coc}$ satisfies the ``partial'' \emph{Smith is Huq} condition, meaning that two $S_{coc}$-equivalence relations centralize each other as soon as their normalization commute in the sense of Huq. Note that Theorem \ref{thmhuq} and Theorem \ref{thmproto} generalize results in \cite{BMS} and \cite{MFM} on the category of monoids.

\begin{figure}[b]
  \begin{tikzpicture}[font={\small}, yscale=1.5, xscale=1.4]
	\begin{pgfonlayer}{nodelayer}
		\node [style=green] (0) at (-8.75, 5.5) {$Z \otimes C$};
		\node [style=green] (1) at (-3, 5.5) {$A^2$};
		\node [style=none] (4) at (5, 6) {$\hat{\kappa} \cdot \hat{\lambda}$};
		\node [style=none] (5) at (-0.75, 6) {$m$};
		\node [style=none] (6) at (-5.75, 6) {$\hat{\kappa} \otimes e \cdot e'$};
		\node [style=green] (7) at (2.75, 5.5) {$A$};
		\node [style=green] (9) at (11.5, 5.5) {$A$};
		\node [style=green] (10) at (2.75, 3.25) {$A^2$};
		\node [style=green] (11) at (11.5, 3.25) {$A^2$};
		\node [style=none] (12) at (3.25, 4.5) {$\Delta $};
		\node [style=none] (13) at (6.75, 3.75) {$\kappa\cdot \lambda \otimes e \cdot \kappa' \cdot \lambda' \cdot \alpha $};
		\node [style=none] (14) at (11, 4.5) {$m$};
		\node [style=green] (15) at (0, 1.25) {$A^4$};
		\node [style=green] (16) at (0, 4) {$A^4$};
		\node [style=none] (17) at (2.5, 2.5) {$m^2$};
		\node [style=none] (18) at (-1.75, 4.25) {$\Delta^2$};
		\node [style=green] (20) at (-6.75, 4) {$Z^2 \otimes C^2$};
		\node [style=green] (21) at (-6.75, 1.25) {$(Z \otimes C)^2$};
		\node [style=none] (22) at (-3.5, 1.75) {$(\hat{\kappa} \otimes e \cdot e')^2$};
		\node [style=none] (24) at (-7, 5) {$\Delta^2$};
		\node [style=green] (25) at (-6.5, 2.5) {$1_Z \otimes \sigma \otimes 1_C $};
		\node [style=green] (26) at (-0.75, 3) {$1_A \otimes \sigma \otimes 1_A $};
		\node [style=green] (27) at (7.25, 1.25) {$A \otimes B$};
		\node [style=none] (29) at (5.75, 0.75) {$m^2$};
		\node [style=green] (30) at (8.5, 2.25) {$\kappa\cdot \lambda \otimes e \cdot \kappa' \cdot \lambda' $};
		\node [style=green] (31) at (4.25, 1.25) {$A^2 \otimes B^2$};
		\node [style=none] (32) at (2.25, 1.75) {$1_A^2 \otimes \alpha^2$};
		\node [style=green] (35) at (-6.75, -3.25) {$(Z \otimes C)^2$};
		\node [style=green] (36) at (5.5, -1.25) {$A^2 \otimes B^2$};
		\node [style=green] (40) at (6.5, 0) {$m^2$};
		\node [style=green] (42) at (0, -0.25) {$\hat{\kappa} \otimes e \cdot e' \otimes \alpha \cdot \hat{\kappa} \otimes e'$};
		\node [style=green] (43) at (9, -0.75) {$A \otimes Y$};
		\node [style=green] (45) at (9, -3.25) {$A^2 \otimes Y$};
		\node [style=none] (47) at (2.5, -2.75) {$\hat{\kappa} \otimes e \cdot e' \otimes  \overline{\alpha \cdot \hat{\kappa}} \otimes \epsilon_C$};
		\node [style=green] (48) at (-8.75, 0) {$Z^2 \otimes C$};
		\node [style=green] (49) at (-8.75, -5) {$Z \otimes C \otimes Z$};
		\node [style=green] (51) at (-8.5, -2.5) {$1_Z \otimes \sigma$};
		\node [style=green] (52) at (-8.75, 3.25) {$\Delta \otimes 1_C$};
		\node [style=green] (53) at (-1, -4.25) {$\hat{\kappa} \otimes e \cdot e' \otimes  \overline{\alpha \cdot \hat{\kappa}}$};
		\node [style=green] (54) at (11.5, -5) {$A^2 \otimes B$};
		\node [style=none] (55) at (6.5, -4.5) {$\hat{\kappa} \otimes e \cdot e' \otimes  \alpha \cdot \hat{\kappa}$};
		\node [style=green] (56) at (11.5, -2.25) {$A \otimes B$};
		\node [style=green] (57) at (11.25, -3.75) {$m \otimes 1_B$};
		\node [style=green] (58) at (11.25, -0.5) {$\kappa\cdot \lambda \otimes e  $};
		\node [style=green] (59) at (9.75, 0.5) {$\kappa\cdot \lambda \otimes e \cdot \kappa' $};
		\node [style=green] (60) at (-0.5, -1.75) {$\hat{\kappa} \otimes e \cdot e' \otimes \kappa' \cdot \overline{\alpha \cdot \hat{\kappa}} \otimes e'$};
		\node [style=green] (61) at (9, -1.75) {$m \otimes 1_Y$};
		\node [style=green] (62) at (-8.75, -7.75) {$Z \otimes C$};
		\node [style=green] (63) at (-5, -7.75) {$Z^2 \otimes C$};
		\node [style=green] (65) at (-3, -7.25) {$1_Z \otimes \sigma$};
		\node [style=green] (67) at (-0.75, -7.75) {$ Z \otimes C \otimes  Z$};
		\node [style=none] (69) at (3.25, -7.25) {$\hat{\kappa} \otimes e \cdot e' \otimes  \alpha \cdot \hat{\kappa}$};
		\node [style=green] (70) at (11.5, -19.5) {$A$};
		\node [style=green] (71) at (11.5, -18) {$A^2$};
		\node [style=green] (72) at (6, -7.75) {$A^2 \otimes B$};
		\node [style=green] (73) at (11.5, -7.75) {$A \otimes B$};
		\node [style=green] (75) at (11.5, -17) {$\kappa \otimes e  $};
		\node [style=none] (76) at (11, -18.75) {$m$};
		\node [style=green] (77) at (11.5, -16) {$X \otimes B$};
		\node [style=green] (78) at (11.25, -9.75) {$\lambda \otimes 1_B$};
		\node [style=green] (79) at (6, -10) {$A^3 \otimes B$};
		\node [style=green] (80) at (6, -12) {$A^3 \otimes B$};
		\node [style=green] (81) at (6, -14) {$A^2 \otimes B$};
		\node [style=green] (82) at (6, -16) {$X^2 \otimes B$};
		\node [style=none] (83) at (9.25, -15.5) {$m \otimes 1_B$};
		\node [style=none] (84) at (5, -15) {$\lambda^2 \otimes 1_B$};
		\node [style=green] (86) at (6, -13) {$1_A \otimes m \otimes 1_B$};
		\node [style=green] (87) at (6, -10.75) {$1_A \otimes e \cdot \alpha \otimes \kappa \cdot \lambda \otimes 1_B$};
		\node [style=green] (88) at (6, -8.75) {$\Delta \otimes 1_A \otimes 1_B$};
		\node [style=green] (89) at (-0.75, -10) {$A \otimes B$};
		\node [style=green] (90) at (-0.75, -12) {$A^2 \otimes B$};
		\node [style=green] (91) at (-0.75, -8.75) {$\hat{\kappa} \otimes \epsilon_C  \otimes  \alpha \cdot \hat{\kappa}$};
		\node [style=none] (92) at (-2, -11) {$\Delta  \otimes 1_B$};
		\node [style=none] (93) at (2.75, -11.5) {$1_A \otimes e \cdot \alpha \otimes u_A \otimes 1_B$};
		\node [style=green] (94) at (-6, -10) {$Z^2$};
		\node [style=green] (95) at (-7.5, -9) {$\Delta \otimes \epsilon_C$};
		\node [style=none] (97) at (-3.5, -9.5) {$\hat{\kappa}   \otimes  \alpha \cdot \hat{\kappa}$};
		\node [style=green] (98) at (1.75, -13) {$1_A \otimes e \cdot \alpha  \otimes 1_B$};
		\node [style=green] (99) at (-0.75, -16) {$X \otimes B$};
		\node [style=green] (100) at (-0.75, -14.25) {$\lambda \otimes \epsilon_A \otimes 1_B$};
		\node [style=none] (101) at (2.25, -15.5) {$1_X \otimes u_X \otimes 1_B$};
		\node [style=green] (104) at (-8.5, -12.5) {$1_Z \otimes \epsilon_C$};
		\node [style=green] (106) at (-8.75, -16) {$Z $};
		\node [style=green] (107) at (-8.75, -19.5) {A};
		\node [style=green] (109) at (-4, -16) {$A^2$};
		\node [style=none] (111) at (-2.75, -15.5) {$\lambda \otimes \alpha$};
		\node [style=none] (112) at (-6.5, -17.25) {$\Delta$};
		\node [style=none] (113) at (-0.75, -17.75) {$(3)$};
		\node [style=none] (114) at (0.75, -19) {$1_A$};
		\node [style=none] (115) at (5.25, -18.25) {$1_X \otimes 1_B$};
		\node [style=none] (116) at (8.5, -7.25) {$m \otimes 1_B$};
		\node [style=green] (117) at (-8.75, -17.75) {$\hat{\kappa}$};
		\node [style=none] (118) at (6.25, 4.75) {$\eqref{kernelaa'}$};
		\node [style=none] (119) at (1, 4.5) {$\eqref{delta et m}$};
		\node [style=none] (120) at (-7.5, -4.25) {$\eqref{counital comultiplication}$};
		\node [style=none] (121) at (4.5, 0.25) {$(1)$};
		\node [style=none] (122) at (6.75, -2) {$(4)$};
		\node [style=none] (123) at (-7, -7.25) {$\Delta \otimes 1_C$};
		\node [style=none] (124) at (-5.25, -13) {$\eqref{counital comultiplication}$};
		\node [style=none] (125) at (5.25, -17) {$\eqref{unital multiplication}$};
		\node [style=none] (126) at (2.5, -14.25) {$(2)$};
		\node [style=none] (127) at (3.25, -12.5) {$\eqref{unital multiplication}$};
		\node [style=none] (128) at (2, -10) {$(2)$};
		\node [style=none] (129) at (9, -12) {$\eqref{lambda morph Prop 2.5}$};
		\node [style=none] (130) at (-3.25, -1) {$\eqref{factorise kernel}$};
		\node [style=green] (131) at (-6.5, -0.75) {$(1_Z \otimes 1_C)^2$};
	\end{pgfonlayer}
	\begin{pgfonlayer}{edgelayer}
		\draw [style=stealth] (0) to (1);
		\draw [style=stealth] (7) to (9);
		\draw [style=stealth] (10) to (11);
		\draw [style=stealth] (11) to (9);
		\draw [style=stealth] (7) to (10);
		\draw [style=stealth] (1) to (7);
		\draw [style=stealth] (1) to (16);
		\draw [style=stealth] (16) to (15);
		\draw [style=stealth] (15) to (10);
		\draw [style=stealth] (0) to (20);
		\draw [style=stealth] (20) to (21);
		\draw [style=stealth] (21) to (15);
		\draw [style=stealth] (27) to (11);
		\draw [style=stealth] (21) to (35);
		\draw [style=stealth] (35) to (36);
		\draw [style=stealth] (36) to (27);
		\draw [style=stealth] (21) to (36);
		\draw [style=stealth] (43) to (11);
		\draw [style=stealth] (45) to (43);
		\draw [style=stealth] (35) to (45);
		\draw [style=stealth] (0) to (48);
		\draw [style=stealth] (48) to (49);
		\draw [style=stealth] (49) to (45);
		\draw [style=stealth] (49) to (54);
		\draw [style=stealth] (54) to (56);
		\draw [style=stealth] (56) to (11);
		\draw [style=stealth] (62) to (63);
		\draw [style=stealth] (63) to (67);
		\draw [style=stealth] (71) to (70);
		\draw [style=stealth] (72) to (73);
		\draw [style=stealth] (67) to (72);
		\draw [style=stealth] (73) to (77);
		\draw [style=stealth] (77) to (71);
		\draw [style=stealth] (82) to (77);
		\draw [style=stealth] (81) to (82);
		\draw [style=stealth] (72) to (79);
		\draw [style=stealth] (79) to (80);
		\draw [style=stealth] (80) to (81);
		\draw [style=stealth] (67) to (89);
		\draw [style=stealth] (89) to (90);
		\draw [style=stealth] (90) to (80);
		\draw [style=stealth] (62) to (94);
		\draw [style=stealth] (94) to (89);
		\draw [style=stealth] (90) to (81);
		\draw [style=stealth] (90) to (99);
		\draw [style=stealth] (99) to (82);
		\draw [style=stealth] (106) to (107);
		\draw [style=stealth] (62) to (106);
		\draw [style=stealth] (107) to (109);
		\draw [style=stealth] (109) to (99);
		\draw [style=stealth, bend right] (99) to (77);
		\draw [style=stealth] (107) to (70);
		\draw [style=stealth] (15) to (31);
		\draw [style=stealth] (31) to (27);
	\end{pgfonlayer}
\end{tikzpicture}
 \caption{Condition (4) of the composition of two split extensions}
 \label{(4)composition}
\end{figure}

\begin{landscape}
\begin{figure}
\begin{tikzpicture}[font=\small,xscale=1.8, yscale=2.1]
	\begin{pgfonlayer}{nodelayer}
		\node [style=green] (0) at (-10.75, 6.25) {$(A \otimes_B A)^2$};
		\node [style=green] (1) at (-2.75, 6.25) {$A^4$};
		\node [style=green] (2) at (9.75, 6.25) {$A^2$};
		\node [style=green] (3) at (5, 6.25) {$A^4$};
		\node [style=green] (5) at (9.75, -7) {$A$};
		\node [style=none] (6) at (-6.5, 6.75) {$\varepsilon \otimes \varepsilon$};
		\node [style=none] (7) at (0.25, 6.75) {$1_A \otimes \sigma \otimes 1_A$};
		\node [style=none] (8) at (6.75, 6.75) {$m^2$};
		\node [style=green] (10) at (9.25, 4) {$\kappa \cdot \lambda \otimes 1_A$};
		\node [style=green] (11) at (9.75, -2.75) {$m$};
		\node [style=green] (12) at (5, 4.25) {$A^5$};
		\node [style=green] (13) at (5, 2.25) {$A^5$};
		\node [style=green] (14) at (5, 0.25) {$A^4$};
		\node [style=green] (16) at (7.5, 0.25) {$A^4$};
		\node [style=green] (17) at (5, 5.5) {$\Delta \otimes 1_A^3$};
		\node [style=green] (18) at (3.75, 3.25) {$1_A \otimes e \cdot \alpha \otimes \kappa \cdot \lambda \otimes 1_A^2$};
		\node [style=green] (19) at (4.75, 1.5) {$1_A \otimes m \otimes 1_A^2$};
		\node [style=none] (20) at (6.25, 0.75) {$(\kappa \cdot \lambda)^2 \otimes 1_A^2$};
		\node [style=none] (21) at (8.5, 0.75) {$m^2$};
		\node [style=green] (22) at (-2.75, 4.5) {$A^5$};
		\node [style=green] (23) at (-2.75, 2.25) {$A^5$};
		\node [style=none] (24) at (0.5, 1.75) {$1_A^2 \otimes \sigma \otimes 1_A$};
		\node [style=green] (25) at (-2.75, 5.5) {$\Delta \otimes 1_A^3$};
		\node [style=green] (26) at (-2.5, 3.75) {$1_A \otimes e \cdot \alpha \otimes 1_A \otimes \kappa \cdot \lambda \otimes 1_A$};
		\node [style=green] (29) at (-10.75, 5.5) {$\varepsilon \otimes \varepsilon$};
		\node [style=green] (30) at (-8.75, 2.25) {$A^5$};
		\node [style=none] (33) at (-5.75, 2.75) {$1_A \otimes e \cdot \beta \otimes 1_A \otimes \kappa \cdot \lambda \otimes 1_A$};
		\node [style=green] (34) at (-6.75, 0.25) {$A^6$};
		\node [style=green] (36) at (-2.75, 0.25) {$A^6$};
		\node [style=none] (37) at (-2.75, -0.25) {$1_A \otimes e \cdot \beta\otimes \kappa' \cdot \lambda'  \otimes e \cdot \beta \otimes   \kappa \cdot \lambda \otimes 1_A$};
		\node [style=green] (38) at (-7.5, 1) {$1_A^2 \otimes \Delta \otimes 1_A^2$};
		\node [style=green] (39) at (-3.25, 1.25) {$1_A^2 \otimes m \otimes 1_A^2$};
		\node [style=green] (40) at (1, 0.25) {$A^6$};
		\node [style=none] (41) at (3.25, 0.75) {$1_A \otimes m^2 \otimes 1_A$};
		\node [style=none] (43) at (-0.5, 0.75) {$1_A^2 \otimes \sigma_{A^2,A} \otimes 1_A$};
		\node [style=green] (45) at (-5.75, -1.25) {$1_A^2 \otimes \sigma_{A^2,A} \otimes 1_A$};
		\node [style=none] (48) at (-2.75, -2.75) {$1_A \otimes e \cdot \beta \otimes   \kappa \cdot \lambda \otimes \kappa' \cdot \lambda'  \otimes e \cdot \beta  \otimes 1_A$};
		\node [style=green] (50) at (-1.5, -2.25) {$A^6$};
		\node [style=none] (51) at (0, -1.75) {$1_A \otimes m \otimes 1_A^2$};
		\node [style=green] (52) at (1.75, -2.25) {$A^5$};
		\node [style=green] (53) at (5.25, -2.25) {$A^5$};
		\node [style=none] (54) at (4, -1.75) {$(\kappa \cdot \lambda)^2 \otimes 1_A^3$};
		\node [style=green] (55) at (5.75, -0.75) {$1_A^2 \otimes m \otimes 1_A$};
		\node [style=green] (58) at (9.75, 0.25) {$A^2$};
		\node [style=green] (59) at (6.75, -3.75) {$A^4$};
		\node [style=none] (60) at (8.25, -1.25) {$m^2$};
		\node [style=green] (63) at (-6.75, -2.25) {$A^6$};
		\node [style=green] (65) at (-6.75, -4.75) {$A^6$};
		\node [style=green] (67) at (1.5, -4.75) {$A^5$};
		\node [style=green] (69) at (-4.75, -5.25) {$1_A \otimes m \otimes 1_A^3$};
		\node [style=green] (70) at (-4.5, -3.75) {$1_A \otimes e \cdot \beta \otimes   \kappa \cdot \lambda \otimes \kappa' \cdot \lambda'  \otimes e \cdot \beta  \otimes 1_A$};
		\node [style=none] (71) at (-0.5, -5.25) {$(\kappa \cdot \lambda)^2 \otimes 1_A^3$};
		\node [style=green] (72) at (5, -4.75) {$1_A \otimes m \otimes 1_A^2$};
		\node [style=green] (77) at (-9, -2.25) {$1_A \otimes \sigma^2 \otimes 1_A$};
		\node [style=green] (79) at (-2.75, -4.75) {$A^5$};
		\node [style=green] (82) at (3.5, -7) {$A^5$};
		\node [style=green] (84) at (-2.75, -7) {$A^5$};
		\node [style=green] (85) at (0.25, -6.5) {$(\kappa \cdot \lambda \otimes 1_A)^2 \otimes 1_A$};
		\node [style=green] (86) at (-2.75, -6) {$1_A \otimes \sigma \otimes 1_A^2$};
		\node [style=green] (87) at (-7.75, -5.75) {$1_A \otimes \sigma_{A^2,A} \otimes 1_A^2$};
		\node [style=green] (88) at (-10.75, -7) {$A^6$};
		\node [style=green] (89) at (-6.5, -6.5) {$1_A^2 \otimes m \otimes 1_A^2$};
		\node [style=green] (90) at (-2.75, -7) {$A^5$};
		\node [style=green] (91) at (-10, -4.5) {$1_A \otimes \kappa' \cdot \lambda' \otimes e \cdot \beta$ };
		\node [style=green] (92) at (2, -5.75) {$1_A \otimes \sigma \otimes 1_A^2$};
		\node [style=green] (93) at (7.25, -7) {$A^4$};
		\node [style=none] (94) at (5.5, -6.5) {$1_A^2 \otimes m \otimes 1_A$};
		\node [style=none] (95) at (8.5, -6.5) {$m \cdot m^2$};
		\node [style=green] (166) at (-10.75, 4.5) {$A^4$};
		\node [style=green] (167) at (-10.75, 2.5) {$A^5$};
		\node [style=green] (168) at (-11.25, 3.5) {$1_A \otimes \Delta \otimes 1_A^2$};
		\node [style=green] (169) at (-10.75, 0.5) {$A^6$};
		\node [style=green] (170) at (-10.5, 1.5) {$1_A^2 \otimes \Delta \otimes 1_A^2$};
		\node [style=green] (171) at (-10.75, -3.5) {$A^6$};
		\node [style=green] (172) at (-10.25, -0.75) {$1_A^3 \otimes \sigma \otimes 1_A$};
		\node [style=green] (173) at (-10, -5) {$\otimes  \kappa \cdot \lambda \otimes e \cdot \beta \otimes 1_A$ };
		\node [style=none] (174) at (-8.5, 3.5) {$1_A \otimes \Delta \otimes 1_A^2$};
		\node [style=none] (175) at (7.25, 4.75) {$\eqref{lambda morph Prop 2.5}$};
		\node [style=none] (176) at (-7.5, 5) {$\eqref{eq compo}$};
		\node [style=none] (177) at (-5.75, 1.5) {$(3)$};
		\node [style=none] (179) at (-9.5, 0) {$\eqref{coco}$};
		\node [style=none] (180) at (7.5, -1.75) {$\eqref{ass}$};
		\node [style=none] (181) at (8.25, -5.5) {$\eqref{ass}$};
		\node [style=green] (182) at (5.25, -3) {$1_A \otimes m \otimes 1_A^2$};
		\node [style=none] (183) at (2.25, -3.75) {$[X,X'] = 0$};
		\node [style=none] (184) at (-9.25, -3.5) {$\eqref{symm}$};
	\end{pgfonlayer}
	\begin{pgfonlayer}{edgelayer}
		\draw [style=stealth] (0) to (1);
		\draw [style=stealth] (1) to (3);
		\draw [style=stealth] (3) to (2);
		\draw [style=stealth] (3) to (12);
		\draw [style=stealth] (12) to (13);
		\draw [style=stealth] (13) to (14);
		\draw [style=stealth] (14) to (16);
		\draw [style=stealth] (1) to (22);
		\draw [style=stealth] (22) to (23);
		\draw [style=stealth] (23) to (13);
		\draw [style=stealth] (30) to (23);
		\draw [style=stealth] (30) to (34);
		\draw [style=stealth] (34) to (36);
		\draw [style=stealth] (36) to (23);
		\draw [style=stealth] (36) to (40);
		\draw [style=stealth] (40) to (14);
		\draw [style=stealth] (52) to (53);
		\draw [style=stealth] (50) to (52);
		\draw [style=stealth] (53) to (16);
		\draw [style=stealth] (59) to (58);
		\draw [style=stealth] (53) to (59);
		\draw [style=stealth] (2) to (58);
		\draw [style=stealth] (58) to (5);
		\draw [style=stealth] (16) to (58);
		\draw [style=stealth] (63) to (65);
		\draw [style=stealth] (34) to (63);
		\draw [style=stealth] (63) to (50);
		\draw [style=stealth] (79) to (67);
		\draw [style=stealth] (65) to (79);
		\draw [style=stealth] (84) to (82);
		\draw [style=stealth] (79) to (84);
		\draw [style=stealth] (67) to (82);
		\draw [style=stealth] (82) to (59);
		\draw [style=stealth] (88) to (90);
		\draw [style=stealth] (65) to (88);
		\draw [style=stealth] (82) to (93);
		\draw [style=stealth] (93) to (5);
		\draw [style=stealth] (166) to (167);
		\draw [style=stealth] (167) to (169);
		\draw [style=stealth] (169) to (171);
		\draw [style=stealth] (0) to (166);
		\draw [style=stealth] (166) to (30);
		\draw [style=stealth] (34) to (171);
		\draw [style=stealth] (171) to (88);
	\end{pgfonlayer}
\end{tikzpicture}
\caption{The morphism $c$ is a morphism of bialgebras (part 1)}
\label{cmorph1}
\end{figure}
\end{landscape}

\begin{landscape}
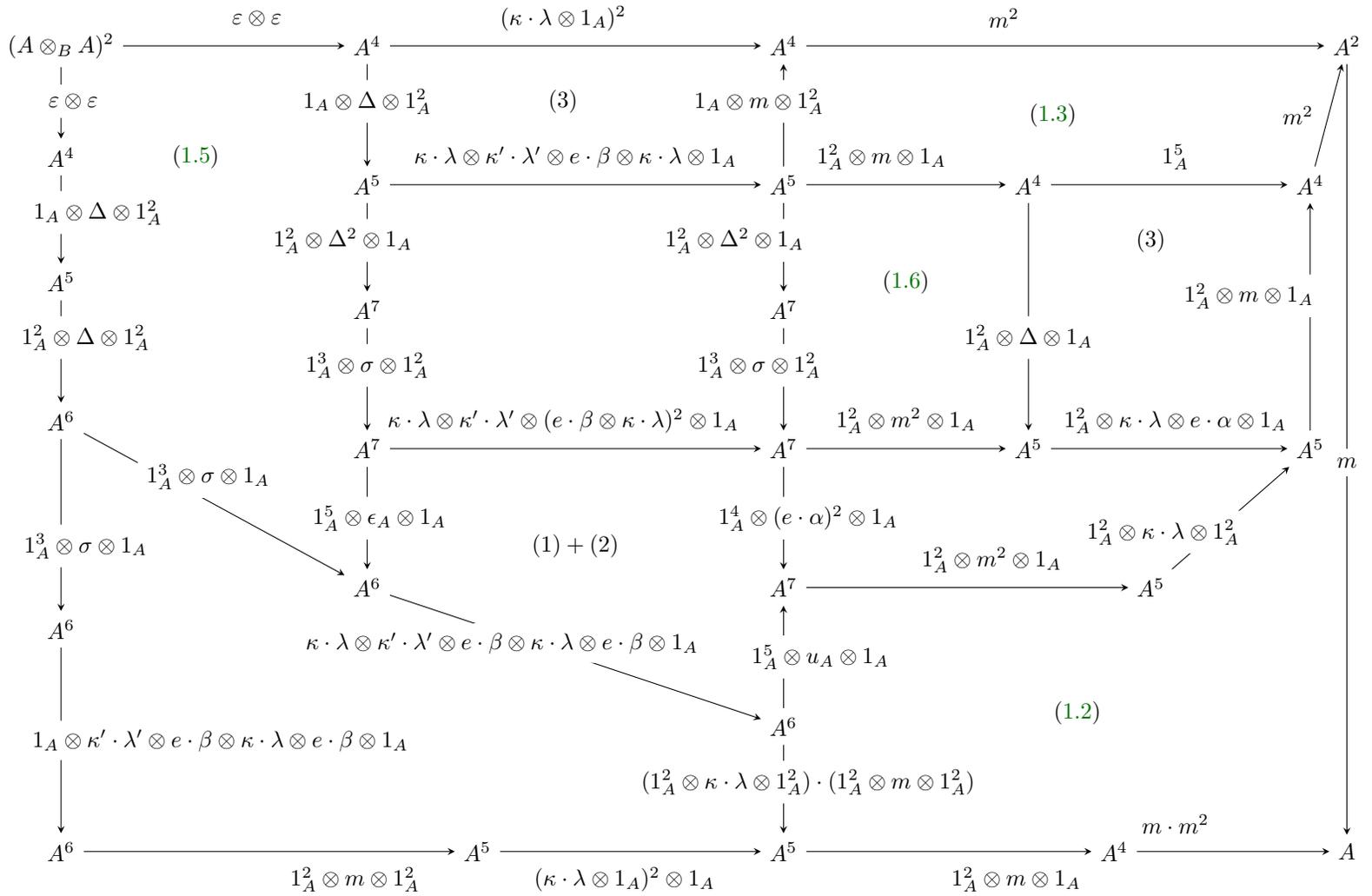
\begin{figure}
\begin{tikzpicture}[font=\small,xscale=1.5, yscale=1.7]
	\begin{pgfonlayer}{nodelayer}
		\node [style=green] (0) at (-6.25, -3.75) {$A^4$};
		\node [style=green] (1) at (-6, -2.75) {$\varepsilon \otimes \varepsilon$};
		\node [style=green] (2) at (-6.25, -6) {$A^5$};
		\node [style=green] (3) at (-5.5, -4.75) {$1_A \otimes \Delta \otimes 1_A^2$};
		\node [style=green] (4) at (-6.25, -8.5) {$A^6$};
		\node [style=green] (5) at (-5.75, -7) {$1_A^2 \otimes \Delta \otimes 1_A^2$};
		\node [style=green] (6) at (-6.25, -12.25) {$A^6$};
		\node [style=green] (7) at (-5.75, -10.75) {$1_A^3 \otimes \sigma \otimes 1_A$};
		\node [style=green] (8) at (-3, -14.25) {$1_A \otimes \kappa' \cdot \lambda' \otimes e \cdot \beta \otimes \kappa \cdot \lambda \otimes e \cdot \beta \otimes 1_A$};
		\node [style=green] (9) at (20, -16.25) {$A$};
		\node [style=green] (10) at (8.5, -16.25) {$A^5$};
		\node [style=green] (11) at (5.25, -16.75) {$(\kappa \cdot \lambda \otimes 1_A)^2 \otimes 1_A$};
		\node [style=green] (12) at (-6.25, -16.25) {$A^6$};
		\node [style=green] (13) at (-0.25, -16.75) {$1_A^2 \otimes m \otimes 1_A^2$};
		\node [style=green] (14) at (2.25, -16.25) {$A^5$};
		\node [style=green] (15) at (15.25, -16.25) {$A^4$};
		\node [style=none] (16) at (13.25, -16.75) {$1_A^2 \otimes m \otimes 1_A$};
		\node [style=none] (17) at (16.5, -15.75) {$m \cdot m^2$};
		\node [style=green] (18) at (-6.25, -1.75) {$(A \otimes_B A)^2$};
		\node [style=green] (19) at (0, -1.75) {$A^4$};
		\node [style=none] (20) at (-2.25, -1.25) {$\varepsilon \otimes \varepsilon$};
		\node [style=green] (21) at (20, -1.75) {$A^2$};
		\node [style=green] (22) at (8.5, -1.75) {$A^4$};
		\node [style=none] (23) at (4, -1.25) {$(\kappa \cdot \lambda \otimes 1_A)^2$};
		\node [style=none] (24) at (13, -1.25) {$m^2$};
		\node [style=green] (25) at (20, -9.25) {$m$};
		\node [style=green] (26) at (8.5, -4.25) {$A^5$};
		\node [style=green] (27) at (0, -2.75) {$1_A \otimes \Delta \otimes 1_A^2$};
		\node [style=green] (28) at (8, -2.75) {$1_A  \otimes m \otimes 1_A^2$};
		\node [style=none] (30) at (4.25, -3.75) {$\kappa \cdot \lambda \otimes \kappa' \cdot \lambda' \otimes e \cdot \beta \otimes \kappa \cdot \lambda \otimes 1_A$};
		\node [style=green] (31) at (13.5, -4.25) {$A^4$};
		\node [style=none] (32) at (16.5, -3.75) {$1_A^5$};
		\node [style=none] (33) at (10.5, -3.75) {$1_A^2  \otimes m \otimes 1_A$};
		\node [style=green] (34) at (19, -3) {$m^2$};
		\node [style=green] (35) at (19.25, -4.25) {$A^4$};
		\node [style=green] (36) at (13.5, -9) {$A^5$};
		\node [style=green] (37) at (19.25, -9) {$A^5$};
		\node [style=green] (38) at (13.5, -7) {$1_A^2 \otimes \Delta \otimes 1_A$};
		\node [style=green] (39) at (18, -6.25) {$1_A^2  \otimes m \otimes 1_A$};
		\node [style=none] (40) at (16.5, -8.5) {$1_A^2  \otimes \kappa \cdot \lambda \otimes e \cdot \alpha \otimes 1_A$};
		\node [style=green] (41) at (8.5, -6.5) {$A^7$};
		\node [style=green] (42) at (8.5, -9) {$A^7$};
		\node [style=none] (43) at (11, -8.5) {$1_A^2  \otimes m^2 \otimes 1_A$};
		\node [style=green] (44) at (7.5, -5.25) {$1_A^2 \otimes \Delta^2 \otimes 1_A$};
		\node [style=green] (45) at (8, -7.5) {$1_A^3 \otimes \sigma \otimes 1_A^2$};
		\node [style=green] (46) at (0, -4.25) {$A^5$};
		\node [style=green] (47) at (0, -6.5) {$A^7$};
		\node [style=green] (48) at (0, -9) {$A^7$};
		\node [style=green] (49) at (-0.5, -5.25) {$1_A^2 \otimes \Delta^2 \otimes 1_A$};
		\node [style=green] (50) at (0, -7.5) {$1_A^3 \otimes \sigma \otimes 1_A^2$};
		\node [style=none] (51) at (4, -8.5) {$\kappa \cdot \lambda \otimes \kappa' \cdot \lambda' \otimes (e \cdot \beta \otimes \kappa \cdot \lambda)^2 \otimes 1_A$};
		\node [style=green] (52) at (8.5, -11.5) {$A^7$};
		\node [style=green] (53) at (16, -11.5) {$A^5$};
		\node [style=none] (54) at (12.75, -11) {$1_A^2  \otimes m^2 \otimes 1_A$};
		\node [style=green] (55) at (9, -10.25) {$1_A^4 \otimes (e \cdot \alpha)^2 \otimes 1_A$};
		\node [style=green] (56) at (16.25, -10.5) {$1_A^2 \otimes \kappa \cdot \lambda \otimes 1_A^2$};
		\node [style=green] (57) at (0.25, -10.25) {$1_A^5 \otimes \epsilon_A \otimes 1_A$};
		\node [style=green] (58) at (0, -11.5) {$A^6$};
		\node [style=green] (59) at (-3.25, -9.5) {$1_A^3\otimes \sigma \otimes 1_A$};
		\node [style=green] (60) at (8.5, -14) {$A^6$};
		\node [style=green] (61) at (9.25, -12.75) {$1_A^5 \otimes u_A \otimes 1_A$};
		\node [style=green] (62) at (9, -15) {$(1_A^2 \otimes \kappa \cdot \lambda \otimes 1_A^2) \cdot (1_A^2  \otimes m \otimes 1_A^2)$};
		\node [style=green] (63) at (2.75, -12.5) {$\kappa \cdot \lambda \otimes \kappa' \cdot \lambda' \otimes e \cdot \beta \otimes \kappa \cdot \lambda \otimes e \cdot \beta \otimes 1_A$};
		\node [style=none] (64) at (14, -3) {$\eqref{ass}$};
		\node [style=none] (65) at (16, -5.25) {$(3)$};
		\node [style=none] (66) at (11, -6) {$\eqref{delta et m}$};
		\node [style=none] (67) at (4, -2.75) {$(3)$};
		\node [style=none] (68) at (-3.5, -3.75) {$\eqref{counital comultiplication}$};
		\node [style=none] (69) at (4.25, -10.75) {$(1) + (2)$};
		\node [style=none] (70) at (14.5, -13.75) {$\eqref{unital multiplication}$};
	\end{pgfonlayer}
	\begin{pgfonlayer}{edgelayer}
		\draw [style=stealth] (0) to (2);
		\draw [style=stealth] (2) to (4);
		\draw [style=stealth] (4) to (6);
		\draw [style=stealth] (12) to (14);
		\draw [style=stealth] (10) to (15);
		\draw [style=stealth] (15) to (9);
		\draw [style=stealth] (6) to (12);
		\draw [style=stealth] (18) to (19);
		\draw [style=stealth] (18) to (0);
		\draw [style=stealth] (26) to (22);
		\draw [style=stealth] (19) to (22);
		\draw [style=stealth] (22) to (21);
		\draw [style=stealth] (21) to (9);
		\draw [style=stealth] (26) to (31);
		\draw [style=stealth] (36) to (37);
		\draw [style=stealth] (37) to (35);
		\draw [style=stealth] (31) to (36);
		\draw [style=stealth] (31) to (35);
		\draw [style=stealth] (35) to (21);
		\draw [style=stealth] (26) to (41);
		\draw [style=stealth] (41) to (42);
		\draw [style=stealth] (42) to (36);
		\draw [style=stealth] (46) to (47);
		\draw [style=stealth] (47) to (48);
		\draw [style=stealth] (19) to (46);
		\draw [style=stealth] (46) to (26);
		\draw [style=stealth] (48) to (42);
		\draw [style=stealth] (42) to (52);
		\draw [style=stealth] (52) to (53);
		\draw [style=stealth] (53) to (37);
		\draw [style=stealth] (48) to (58);
		\draw [style=stealth] (4) to (58);
		\draw [style=stealth] (58) to (60);
		\draw [style=stealth] (60) to (52);
		\draw [style=stealth] (60) to (10);
		\draw [style=stealth] (14) to (10);
	\end{pgfonlayer}
\end{tikzpicture}
\caption{The morphism $c$ is a morphism of bialgebras (part 2)}
\label{cmorph2}
\end{figure}
\end{landscape}

\bibliographystyle{mybibstyleJL}

\bibliography{these}

\begin{thebibliography}{10}

\bibitem{Agore}
A.~Agore,
\newblock {\em Limits of {C}oalgebras, {B}ialgebras and {H}opf algebras},
\newblock Proc. Amer. Math. Soc. {\textbf{139}} {\textbf{3}} (2011)   855--863.

\bibitem{Bourn}
D.~Bourn,
\newblock {\em Normalization Equivalence, Kernel Equivalence, and Affine
  Categories}, pages 43--62.
\newblock 2006.

\bibitem{Bournpartial}
D.~Bourn,
\newblock {\em Partial mal'tsevness and partial protomodularity},
\newblock arxiv:1507.02886 (2015).

\bibitem{BG}
D.~Bourn and M.~Gran,
\newblock {\em Centrality and normality in protomodular categories},
\newblock Theory Appl. Categ. {\textbf{9}} {\textbf{8}} (2002)   151--165.

\bibitem{BMS}
D.~Bourn, N.~Martins-Ferreira, A.~Montoli  and M.~Sobral,
\newblock {\em Schreier Split Epimorphisms in Monoids and in Semirings},
\newblock Textos de matem{\'a}tica / S{\'e}rie B,
\newblock Departamento de matem{\'a}tica da Universidade de Coimbra, 2013.

\bibitem{BMMS}
D.~Bourn, N.~Martins-Ferreira, A.~Montoli  and M.~Sobral,
\newblock {\em Monoids and pointed {S}-protomodular categories},
\newblock Homology, Homotopy and Applications {\textbf{18}} (01 2016)
  151--172.

\bibitem{BM3x3}
D.~Bourn and A.~Montoli,
\newblock {\em The 3x3 lemma in the ${\Sigma}$-maltsev and
  ${\Sigma}$-protomodular settings. applications to monoids and quandles},
\newblock Homology Homotopy Appl. {\textbf{21}} (2018).

\bibitem{CG}
S.~Caenepeel and I.~Goyvaerts,
\newblock {\em Monoidal {H}om–{H}opf algebras},
\newblock Comm. Algebra {\textbf{39}} {\textbf{6}} (2011)   2216--2240.

\bibitem{CPP}
A.~Carboni, M.C. Pedicchio  and N.~Pirovano,
\newblock {\em Internal graphs and internal groupoids in {M}al'cev categories},
\newblock Proc. Conference Montreal 1991 (1992)   97--109.

\bibitem{GVdL}
X.~Garc{\'i}a~Mart{\'i}nez and T.~Van~der {L}inden,
\newblock {\em A note on split extensions of bialgebras},
\newblock Forum Math. {\textbf{30}} (2017).

\bibitem{GSV}
M.~Gran, F.~Sterck  and J.~Vercruysse,
\newblock {\em A semi-abelian extension of a theorem by {T}akeuchi},
\newblock J. Pure Appl. Algebra {\textbf{223}} {\textbf{10}} (2019)   68--110.

\bibitem{Huq}
S.A. Huq,
\newblock {\em Commutator, nilpotency, and solvability in categories},
\newblock Quart. J. Math. Oxford Ser. {\textbf{19}} {\textbf{2}} (1968)
  363–389.

\bibitem{MS}
A.~Makhlouf and S.~Silvestrov,
\newblock {\em Hom–{L}ie admissible {H}om-coalgebras and {H}om–{H}opf
  algebras},
\newblock arXiv:0709.2413 (2007).

\bibitem{MFM}
N.~Martins-Ferreira and A.~Montoli,
\newblock {\em On the “{S}mith is {H}uq” condition in {S}-protomodular
  categories},
\newblock Appl. Categ. Structures {\textbf{25}} (2015)   59–75.

\bibitem{MFMS}
N.~Martins-Ferreira, A.~Montoli  and M.~Sobral,
\newblock {\em Semidirect products and crossed modules in monoids with
  operations},
\newblock J. Pure Appl. Algebra {\textbf{217}} {\textbf{2}} (2013)   334--347.

\bibitem{MFVdl}
N.~Martins-Ferreira and T.~Van~der Linden,
\newblock {\em A note on the “{S}mith is {H}uq” condition},
\newblock Appl. Categ. Structures {\textbf{20}} (2012)   175–187.

\bibitem{Stercksplit}
F.~Sterck,
\newblock {\em Split extensions and actions of bialgebras and {H}opf algebras},
\newblock https://doi.org/10.1007/s10485-021-09659-5, Appl. Categ. Structures
  (2021).

\end{thebibliography}
\end{document}